\documentclass[11pt]{amsart}
\usepackage{amsfonts}

\newtheorem{theorem}{Theorem}
\theoremstyle{plain}

\newtheorem{corollary}{Corollary}

\newtheorem{definition}{Definition}

\newtheorem{lemma}{Lemma}

\newtheorem{proposition}{Proposition}

\numberwithin{equation}{section}
\numberwithin{theorem}{section}
\numberwithin{equation}{section}

\begin{document}
	\title[The Cauchy Problem in generalized H\"{o}lder spaces]{On the Cauchy problem for integro-differential equations with space-dependent operators in generalized H\"{o}lder classes}
	\author{Fanhui Xu}
	\email{fanhuixu@usc.edu}
	\address{Department of Mathematics, University of Southern California, Los Angeles}
	\date{September 22, 2018}
	
	\begin{abstract}
		Parabolic integro-differential Kolmogorov equations with different space-dependent operators are considered in H\"{o}lder-type spaces defined by a scalable L\'{e}vy measure. Probabilistic representations are used to prove continuity of the operator. Existence and uniqueness of the solution are established and some regularity estimates are obtained.
	\end{abstract}
	
	\maketitle
	\tableofcontents

\section{Introduction}

Let $\left( \Omega,\mathcal{F},\mathbf{P}\right)$ be a complete probability space and  $\nu$ be a L\'{e}vy measure on $\mathbf{R}^d_0=\mathbf{R}^d\backslash\{0\}$ that is of order $\alpha$, i.e.
\begin{eqnarray*}
	\alpha:=\inf\{\sigma\in\left( 0,2\right):\int_{\left\vert y\right\vert\leq 1}\left\vert y\right\vert ^{\sigma}\nu\left( dy\right)<\infty\}.
\end{eqnarray*}
 We denote by $J\left( ds, dy\right)$ a Poisson random measure on $\left( \Omega,\mathcal{F},\mathbf{P}\right)$ such that $\mathbf{E}\left[ J\left( ds, dy\right)\right]= \nu\left( dy\right)ds$, and denote by $Z_t^{\nu}$ the L\'{e}vy process
\begin{eqnarray}\label{lev}
\qquad Z_t^{\nu}=\int_0^t\int_{\mathbf{R}^d_0} \chi_{\alpha}\left( y\right) y \tilde{J}\left( ds, dy\right)+\int_0^t\int_{\mathbf{R}^d_0} \left( 1-\chi_{\alpha}\left( y\right)\right) y J\left( ds, dy\right).
\end{eqnarray}
Here $\chi_{\alpha}\left( y\right):=1_{\alpha\in\left(1,2\right)}+1_{\alpha=1}1_{\left\vert y\right\vert\leq 1}$, and
\begin{eqnarray*}
	\tilde{J}\left( ds, dy\right):= J\left( ds, dy\right)-\nu\left( dy\right)ds
\end{eqnarray*}
is the compensated Poisson measure.

This work is a continuation of \cite{mf}, in which we studied the Cauchy problem for the following parabolic-type Kolmogorov equations in generalized H\"{o}lder spaces $\tilde{C}^{\beta}\left( \mathbf{R}^{d}\right)$ endowed with norms $|\cdot |_{\beta}$ (see Section 2.2):
\begin{eqnarray}\label{eq1}
\partial_t u\left( t,x\right)&=&L^{\nu}u\left( t,x\right)-\lambda u\left( t,x\right)+f\left( t,x\right), \lambda\geq 0,\\
u\left( 0,x\right)&=& 0,\quad\left( t,x\right)\in \left[0,T\right]\times\mathbf{R}^d,\nonumber
\end{eqnarray}
where $L^{\nu}$ is the infinitesimal generator of $Z_t^{\nu}$. Namely, for any $\varphi\in  C_0^{\infty}\left(\mathbf{R}^d\right)$, 
\begin{eqnarray}\label{op}
L^{\nu}\varphi\left( x\right):=\int\left[ \varphi\left( x+y\right)-\varphi\left( x\right)-\chi_{\alpha}\left( y\right)y\cdot \nabla \varphi\left( x\right)\right]\nu\left(d y\right).
\end{eqnarray}
A notion of \textbf{scaling functions} was utilized in \cite{mf} to include some recent popular models of $\nu$ (cf. \cite{kim1,kim2,zh}).
\begin{definition}
	A continuous function $w: \left( 0,\infty\right)\rightarrow \left( 0,\infty\right)$ is called a scaling function if
	\begin{eqnarray*}
		\lim_{r\rightarrow 0}w\left( r\right)=0, \quad\lim_{R\rightarrow \infty}w\left( R\right)=\infty
	\end{eqnarray*}
	and if there is a nondecreasing continuous function $l\left(\varepsilon\right),\varepsilon>0$ such that $\lim_{\varepsilon\rightarrow 0}l\left(\varepsilon\right)=0$ and 
	\begin{equation}\label{scale}
	w\left( \varepsilon r\right)\leq l\left(\varepsilon\right)w\left( r\right), \quad\forall r,\varepsilon>0.
	\end{equation}
$l$ is called the scaling factor of $w$.
\end{definition}

For any L\'{e}vy measure $\nu$, any $R>0$ and $\forall B\in\mathcal{B}\left( \mathbf{R}^d_0\right)$, 
\begin{eqnarray}
\nu_R \left(B\right)&:=&\int 1_B\left( y/R\right)\nu\left( dy\right),\label{mea}\\ 
\tilde{\nu}_R\left( dy\right) &:=&w\left( R\right) \nu_R\left( dy\right).\label{meas}
\end{eqnarray}
We can always normalize $w$ by a constant so that $w\left( 1\right)=1$ and $\tilde{\nu}_1\left( dy\right)=\nu\left( dy\right)$. It was imposed in \cite{mf} for $\nu$:\\
\noindent \textbf{A(w,l)} (i) (Non-degeneracy) Suppose $\tilde{\nu}_R\left( dy\right)\geq \mu^0 \left( dy\right), R>0$ for some L\'{e}vy measure $\mu^0$ that is supported on the unit ball $B\left( 0\right)$, with $\mu_0$ satisfying 
\begin{equation}\label{mu0}
\int \left\vert y\right\vert^2 \mu^0\left( dy\right) +\int \left\vert\xi\right\vert^4\left[1+\upsilon\left( \xi\right)\right]^{d+3}\exp\{-\zeta^0\left( \xi\right) \}d\xi<\infty,
\end{equation} 
where
\begin{eqnarray*}
	\upsilon\left( \xi\right)&=&\int \chi_{\alpha}\left( y\right) \left\vert y\right\vert \left[ \left( \left\vert \xi\right\vert\left\vert y\right\vert\right)\wedge 1\right]\mu^0\left(dy\right),\\
	\quad\zeta^0\left( \xi\right)&=&\int \left[ 1-\cos \left( 2\pi\xi\cdot y\right)\right]\mu^0\left( dy\right).
\end{eqnarray*}
In addition, for all $\xi\in S_{d-1}=\{\xi\in\mathbf{R}^d:\left\vert\xi\right\vert=1 \}$, there is a constant $c_1>0$, such that
\begin{equation*}
\int_{\left\vert y\right\vert\leq 1} \left\vert \xi\cdot y\right\vert^2 \mu^0\left( dy\right)\geq c_0.
\end{equation*} 
(ii) (Symmetry) If $\alpha=1$, then 
\begin{equation}\label{alpha1}
\int_{r<\left\vert y\right\vert <R} y\nu\left( dy\right)=0 \quad\text{ for all } 0<r<R<\infty.
\end{equation}
(iii) (Scalability) There exist constants $\alpha_1\geq\alpha_2$ such that $\alpha_1,\alpha_2\in\left( 0,1\right)$ if $\alpha\in\left( 0,1\right)$, $\alpha_1,\alpha_2\in \left( 1,2\right]$ if $\alpha\in\left( 1,2\right)$, $\alpha_1\in\left( 1,2\right]$ and $\alpha_2\in\left[0,1\right)$ if $\alpha=1$, and 
\begin{equation*}
\int_{\left\vert y\right\vert\leq 1} \left\vert y\right\vert^{\alpha_1}\tilde{\nu}_R\left( dy\right)+\int_{\left\vert y\right\vert> 1} \left\vert y\right\vert^{\alpha_2}\tilde{\nu}_R\left( dy\right)\leq N_0.
\end{equation*}
The $N_0>0$ above is uniform with respect to $R$. \\
\noindent(iv) (Scalability) Suppose $\varsigma\left( r\right):=\nu\left( \left\vert y\right\vert >r\right),r>0$ is continuous in $r$ and
\begin{eqnarray*}
	\int_0^{1}s \varsigma\left( rs\right)\varsigma\left( r\right)^{-1} ds\leq C_0
\end{eqnarray*}
for some positive $C_0$ independent of $r$.

Under \textbf{A(w,l)}, $Z_t^{\nu}$ possesses a smooth density function whose regularity estimates were derived in \cite{cr}. Moreover,  $Z_t^{\nu}$ is approximately distributed as  $\frac{1}{R}Z_{w\left( R\right)t}^{\nu}, R>0$. This property gives a uniform description of L\'{e}vy measures that were considered in \cite{zh}, \cite{kim1} and \cite{kim2}. In \cite{zh}, $\nu$ is assumed to be confined by two $\alpha$-stable measures of the same order, namely, 
\begin{eqnarray}\label{as1}
&&\int_{S_{d-1}}\int_{0}^{\infty }1_{B}\left(
rw\right) \frac{dr}{r^{1+\alpha }}\Sigma _{1}\left( dw\right) \nonumber\\
&\leq &\nu \left( B\right) \leq \int_{S_{d-1}}\int_{0}^{\infty }1_{B}\left( rw\right) \frac{dr}{%
	r^{1+\alpha }}\Sigma _{2}\left( dw\right)
\end{eqnarray}
for any Borel measurable set $B$. They also assumed $\Sigma_1$ and $\Sigma_2$ are two finite measures defined on the unit sphere and $\Sigma_1$ is nondegenerate. In this situation, $\nu$ satisfies \textbf{A(w,l)} with $w\left( r\right)=l\left( r\right)=r^{\alpha},~r>0$. Another interesting class of L\'{e}vy measures was investigated in \cite{kim1} and \cite{kim2}, where
\begin{equation}\label{mod}
\nu\left( B\right)=\int_0^{\infty}\int_{\left\vert w\right\vert=1} 1_B\left( rw\right)a\left( r,w\right)j\left(r\right)r^{d-1}\Sigma\left( dw\right)dr, \quad \forall B\in\mathcal{B}\left( \mathbf{R}^d_0\right),
\end{equation}
$\Sigma\left( dw\right)$ is a finite measure on the unit sphere, and
\begin{equation*}
j\left( r\right)=\int_0^{\infty}\left( 4\pi t\right)^{-d/2}\exp\left( -\frac{r^2}{4t}\right)\Lambda\left( dt\right), r>0,
\end{equation*} 
with $\Lambda\left( dt\right)$ being a measure on $\left(0,\infty\right)$ such that $\int_0^{\infty} \left( 1\wedge t\right) \Lambda\left( dt\right)<\infty$. Let $\phi\left(r\right)=\int_0^{\infty}\left( 1-e^{-rt}\right)\Lambda\left( dt\right), r\geq 0$ be the associated Bernstein function. They imposed\\
\textbf{H. }there is a function $\rho_0\left(w\right)$ defined on the unit sphere such that $\rho_0\left(w\right)\leq a\left( r,w\right)\leq 1, \forall r>0$, and for all $\left\vert \xi\right\vert=1$,
\begin{equation*}
\int_{S^{d-1}}\left\vert \xi\cdot w\right\vert^2\rho_0\left( w\right) \geq c>0.
\end{equation*}
\textbf{G. }(i) There is $C>1$ such that
\begin{equation*}
\frac{1}{C}\phi\left( r^{-2}\right)r^{-d}\leq j\left(r\right)\leq C\phi\left( r^{-2}\right)r^{-d}.
\end{equation*}
(ii) There are $0<\sigma_1\leq \sigma_2<1$ and $C>0$ such that for all $0<r\leq R$
\begin{equation*}
C^{-1}\left( \frac{R}{r}\right)^{\sigma_1}\leq \frac{\phi\left( R\right)}{\phi\left( r\right)}\leq C\left( \frac{R}{r}\right)^{\sigma_2}.
\end{equation*}
It can be verified that \textbf{H} and \textbf{G} produce L\'{e}vy measures of \textbf{A(w,l)}-type with $w\left( r\right)=j\left( r\right)^{-1}r^{-d}, r>0$, and
\begin{equation*}
l\left( r\right)=\left\{\begin{array}{ll}
Cr^{2\sigma_1} & \mbox{ if } r\leq 1,\\
Cr^{2\sigma_2} & \mbox{ if } r> 1
\end{array}\right.
\end{equation*}
for some $C>0$. (See \cite{kim1,kim2,cr,zh} for details and examples.)

Write $H_T=\left[0,T\right]\times\mathbf{R}^d$. In this note, we consider the following parabolic integro-differential equation:
\begin{eqnarray}\label{eq3}
\partial_t u\left( t,x\right)&=&\mathcal{L}u\left( t,x\right)-\lambda u\left( t,x\right)+f\left( t,x\right), ~\lambda\geq 0,\\
u\left( 0,x\right)&=& 0,~\left( t,x\right)\in H_T,\nonumber
\end{eqnarray}
where $\mathcal{L}=\mathcal{A}+\mathcal{Q}$ or $\mathcal{L}=\mathcal{G}+\mathcal{Q}$, and for any function $\varphi\in  C_0^{\infty}\left(\mathbf{R}^d\right)$,
\begin{eqnarray}
	\mathcal{A}\varphi\left( x\right)&:=&\int\left[\varphi\left( x+y\right)-\varphi\left( x\right)-\chi_{\alpha}\left( y\right)y\cdot \nabla \varphi\left( x\right) \right]\rho\left(t,x,y\right)\nu\left(d y\right),\notag\\
	\mathcal{G}\varphi\left( x\right)&:=&\int\left[\varphi\left( x+G\left( x\right)y\right)-\varphi\left( x\right)-\chi_{\alpha}\left( y\right)G\left( x\right)y\cdot \nabla \varphi\left( x\right) \right]\nu\left(d y\right),\notag\\
	\mathcal{Q}\varphi\left( x\right)&:=& 1_{\alpha\in\left( 1,2\right)}b\left( t,x\right)\cdot\nabla \varphi\left( x\right)+p\left( t,x\right)\varphi\left( x\right) +\int_{ \mathbf{R}^d_0}\lbrack \varphi\left( x+q\left( t,x,y\right)\right)\notag\\
	&&-\varphi\left( x\right)-\nabla \varphi\left( x\right)\cdot q\left( t,x,y\right)1_{\alpha\in\left( 1,2\right)}1_{\left\vert y\right\vert\leq 1}\rbrack\varrho\left(t,x,y\right) \nu_2\left( dy\right).\label{low}
\end{eqnarray}

We assume for the underlying L\'{e}vy measure $\nu$:\\
\noindent\textbf{$\mathbf{\tilde{A}}(w,l,\gamma)$}. (i) $\nu$ satisfies \textbf{A(w,l)}.\\
\noindent(ii) There is $\varepsilon\in\left(0,1\right)$ such that for any $\beta'\in\left(0,\beta+\varepsilon\right)$, 
\begin{eqnarray*}
	\int_0^1 l\left( t\right)^{\beta'}\frac{dt}{t}+\int_1^{\infty} l\left( t\right)^{\beta'}\frac{dt}{t^2}+1_{\alpha\in\left[1,2\right) }\int_0^1 l\left( t\right)^{1+\beta'}\frac{dt}{t^2}<\infty.
\end{eqnarray*}
(iii) Set $\gamma(t)=\inf\{s>0:l(s)\geq t\}$ for $t>0$. There exist $0<\delta<\min\left( \frac{1}{2},\beta\right)$ and $0<\delta'<\min\left( \frac{1}{2},\varepsilon\right)$ for the $\varepsilon$ in (ii) such that 
\begin{eqnarray*}
	1_{\alpha \in \left(0,1\right)}\int_1^{\infty}t^{\delta}\gamma\left( t\right)^{-1}dt&<&\infty,\\
	1_{\alpha=1}\left(\int_0^{1}t^{\delta}\gamma\left( t\right)^{-1}dt+\int_1^{\infty}t^{-\delta'}\gamma\left( t\right)^{-1}+t^{\delta}\gamma\left( t\right)^{-2}dt\right)&<&\infty,\\
	1_{\alpha\in \left(1,2\right)}\left(\int_0^{1}t^{-\delta}\gamma\left( t\right)^{-1}dt+\int_1^{\infty}t^{\delta}\gamma\left( t\right)^{-2}+t^{-\frac{1}{2}+\delta}\gamma\left( t\right)^{-1}dt\right)&<&\infty.
\end{eqnarray*}

Suppose the kernel function $\rho$ satisfies\\
\noindent\textbf{H($K,\beta$)}. (i) There is $K>0$ so that for $\forall t\in\left[0,T\right]$,
\begin{eqnarray}
\left\vert \rho\left( t,x, y\right)\right\vert &\leq& K, \quad\forall x, y\in\mathbf{R}^d,\label{bound}\\
\quad\left\vert \rho\left(t, x_1, y\right)-\rho\left(t, x_2, y\right)\right\vert&\leq& Kw\left(\left\vert x_1- x_2\right\vert\right)^{\beta}, \forall y\in\mathbf{R}^d.\qquad\label{bound2}
\end{eqnarray}
(ii) If $\alpha=1$, then for $\forall x\in\mathbf{R}^d, \forall r\in\left( 0,1\right),\forall t\in\left[0,T\right]$,
\begin{eqnarray*}
\int_{r<\left\vert y\right\vert< 1}y\rho\left(t, x,y\right)\nu\left( dy\right)=0.\label{sym}
\end{eqnarray*}

We assume for the main part $\mathcal{G}$:\\
\noindent\textbf{G($c_0, K,\beta$)}. (i) $G\left( z\right), z\in\mathbf{R}^d$ is an invertible and uniform continuous $d\times d$-matrix, and $G\left( z\right)\neq G\left( z'\right)$ if $z\neq z'$.\\
(ii) $\left\vert \det G\left( z\right)\right\vert\geq c_0, \left\Vert G\left( z\right)\right\Vert\leq K, \forall z\in\mathbf{R}^d$ for some $c_0,K>0$.\\
(iii) For the same $K$, $g\left(z,z'\right)\leq K w\left( \left\vert z-z'\right\vert\right)^{\beta}, \forall z,z'\in\mathbf{R}^d$, where $\bar{G}_{z,z'}:=\left\Vert G\left( z\right)-G\left( z'\right)\right\Vert$ and
\begin{eqnarray*}
g\left(z,z'\right)
= \left\{ 
\begin{array}{cc}
w\left(\bar{G}_{z,z'}^{-1}\right)^{-1} & \text{if } \alpha\in\left( 0,1\right) , \\ 
w\left( \bar{G}_{z,z'}^{-1}\right)^{-1}w\left( \bar{G}_{z,z'}\right)^{-\delta'} \vee \bar{G}_{z,z'} & \text{if } \alpha=1, \\ 
\bar{G}_{z,z'} & \text{if } \alpha\in\left( 1,2\right) . 
\end{array}%
\right. 
\end{eqnarray*}

For the same $w,l,K,\beta$, we assume the lower order part $\mathcal{Q}$ satisfies:\\
\noindent\textbf{B($K,\beta$)}. 
(i) $\left\vert b\left( t,\cdot\right)\right\vert_{\beta}+\left\vert p\left( t,\cdot\right)\right\vert_{\beta}+\left\vert \varrho\left( t,\cdot,y\right)\right\vert_{\beta}\leq K,\forall y\in\mathbf{R}^d_0,\forall t\in\left[0,T\right]$.\\
(ii) For all $\alpha\in\left(0,2\right),z'\in\mathbf{R}^d,\forall t\in\left[0,T\right]$, $q\left( t,\cdot,y\right)\neq 0$ if $y\neq 0$. Besides,
\begin{eqnarray*}
	\lim_{\varepsilon\to 0}\sup_{t,z'}\int_{\left\vert q\left( t,z',y\right)\right\vert\leq \varepsilon}\left( w\left( \left\vert q\left( t,z',y\right)\right\vert\right)+1_{\alpha=1}\left\vert q\left( t,z',y\right)\right\vert\right)\nu_2\left( dy\right)=0.
\end{eqnarray*}
(iii) For all $\alpha\in\left(0,1\right],z'\in\mathbf{R}^d,\forall t\in\left[0,T\right]$,
\begin{eqnarray*}
\int_{ \mathbf{R}^d_0 }1_{\alpha<1}\left(w\left( \left\vert q\left( t,z',y\right)\right\vert\right)\wedge 1\right)+1_{\alpha=1}\left(\left\vert q\left( t,z',y\right)\right\vert\wedge 1\right)\nu_2\left( dy\right)&\leq& K.
\end{eqnarray*}
(iv) For all $\alpha\in\left(1,2\right),z'\in\mathbf{R}^d,\forall t\in\left[0,T\right]$,
\begin{eqnarray*}
\int_{ \left\vert y\right\vert\leq 1}w\left( \left\vert q\left( t,z',y\right)\right\vert\right) \nu_2\left( dy\right)+\int_{ \left\vert y\right\vert> 1}\left(\left\vert q\left( t,z',y\right)\right\vert\wedge 1\right)\nu_2\left( dy\right)&\leq& K.
\end{eqnarray*}
(v) For all $z',h\in\mathbf{R}^d,\forall t\in\left[0,T\right]$, $\alpha\in\left( 1,2\right)$,
\begin{eqnarray*}
	\int_{\left\vert y\right\vert\leq 1}w\left(  \left\vert q\left( t,x+h,y\right)\right\vert\right)^{\beta}\left\vert q\left( t,x+h,y\right)-q\left( t,x,y\right)\right\vert\nu_2\left( dy\right)&\leq& K w\left( \left\vert h\right\vert\right)^{\beta},\\
	\int_{\left\vert y\right\vert\leq 1}w\left(  \left\vert q\left( t,x+h,y\right)-q\left( t,x,y\right)\right\vert\right)^{\beta}\left\vert q\left( t,x,y\right)\right\vert\nu_2\left( dy\right)&\leq& K w\left( \left\vert h\right\vert\right)^{\beta},\\
	\int_{\left\vert y\right\vert> 1}\left( \left\vert q\left( t,z'+h,y\right)-q\left( t,z',y\right)\right\vert\wedge 1\right)\nu_2\left( dy\right)&\leq& K w\left( \left\vert h\right\vert\right)^{\beta}.
\end{eqnarray*}
If $\alpha\in\left( 0,1\right)$,
\begin{eqnarray*}
	\int_{ \mathbf{R}^d_0 }\left(w\left( \left\vert q\left( t,z'+h,y\right)-q\left( t,z',y\right)\right\vert\right)\wedge 1\right)\nu_2\left( dy\right)&\leq& K w\left( \left\vert h\right\vert\right)^{\beta}.
\end{eqnarray*}
And if $\alpha=1$,
\begin{eqnarray*}
	\int_{ \mathbf{R}^d_0 }\left( \left\vert q\left( t,z'+h,y\right)-q\left( t,z',y\right)\right\vert\wedge 1\right)\nu_2\left( dy\right)&\leq& K w\left( \left\vert h\right\vert\right)^{\beta}.
\end{eqnarray*}

The main conclusion of this paper is 
\begin{theorem}\label{main}
	Let \textbf{$\mathbf{\tilde{A}}(w,l,\gamma)$}, \textbf{B($K,\beta$)} and \textbf{H($K,\beta$)} (resp. \textbf{G($c_0, K,\beta$)}) hold. If $f\left( t,x\right)\in\tilde{C}^{\beta}\left(H_T\right),\beta\in\left(0,\frac{1}{\alpha}\right)$, then there is a unique solution $u\left( t,x\right)\in \tilde{C}^{1+\beta}\left(H_T\right)$ to $\eqref{eq3}$ with $\mathcal{L}=\mathcal{A}+\mathcal{Q}$ (resp. $\mathcal{L}=\mathcal{G}+\mathcal{Q}$). Moreover, there exists a constant $C$ depending on $c_0,c_1,N_1,K,\beta, d, T$, $\mu,\nu$ such that
	\begin{eqnarray*}
		\left\vert u\right\vert_{\beta}&\leq& C\left( \lambda^{-1}\wedge T\right) \left\vert f\right\vert_{\beta},\label{est7}\\
		\left\vert u\right\vert_{1+\beta}&\leq& C\left\vert f\right\vert_{\beta}.\label{est8}
	\end{eqnarray*}
	And there is a constant $C$ depending on $c_0, c_1,N_1,K,\kappa,\beta,d, T,\mu,\nu$ such that for all $0\leq s<t\leq T$, $\kappa\in\left[ 0,1\right]$ and $\kappa+\beta>1$, 
	\begin{equation*}\label{est9}
	\left\vert u\left(t,\cdot\right)-u\left(s,\cdot\right)\right\vert_{\kappa+\beta}\leq C\left\vert t-s\right\vert^{1-\kappa}\left\vert f\right\vert_{\beta}.
	\end{equation*}
\end{theorem}

Due to generality of the measure $\nu$ we are considering, the L\'{e}vy symbol $\psi^{\nu}\left( \xi\right), \xi\in\mathbf{R}^d$ of the process $Z_t^{\nu}$ is generally not smooth in $\xi$. This was already an obstacle for applying the standard Fourier multiplier theorem to solutions of equations with space-independent coefficients, and it continues to be a difficulty in this work. Thus, probabilistic representations are used instead and continuity of the operators are proved in that approach. Then we apply continuation of parameters, which was also used in \cite{mp} and \cite{mp2}, to show well-posedness of the Cauchy problem. In \cite{mp2}, a parabolic-type Kolmogorov equation with an operator $\mathcal{L}=\mathcal{A}+\mathcal{Q}$ was considered in the standard H\"{o}lder-Zygmund space, where $\mathcal{Q}$ is the lower order part and the principal part
\begin{eqnarray*}
	\mathcal{A}u\left(t, x\right):=\int\left[ u\left(t, x+y\right)-u\left( t,x\right)-\chi_{\alpha}\left( y\right)y\cdot \nabla u\left( t,x\right)\right]\rho\left( t,x,y\right)\frac{dy}{\left\vert y\right\vert^{d+\alpha}}.
\end{eqnarray*}
With more flavor of probability, in \cite{mp} a stochastic parabolic integro-differential equation with operators
\begin{eqnarray*}
	\mathcal{L}u\left(t, x\right)&:=&\int\left[ u\left(t, x+y\right)-u\left( t,x\right)-1_{\alpha\geq 1}1_{\left\vert y\right\vert\leq 1}y\cdot \nabla u\left( t,x\right)\right]\nu\left( t,x,dy\right)\\
	&+&1_{\alpha=2}a^{ij}\left(t, x\right)\partial^2_{ij}u\left(t, x\right)+1_{\alpha\geq 1}\tilde{b}^{i}\left(t, x\right)\partial_i u\left(t, x\right)+l\left(t, x\right) u\left(t, x\right)
\end{eqnarray*}
was studied in H\"{o}lder spaces. A deterministic model with a similar operator was addressed in the little H\"{o}lder-Zygmund spaces in \cite{mp3}. Besides, the Cauchy problem for a second order linear SPDE was considered in \cite{rm} and \cite{br} in standard H\"{o}lder classes.

The outline of this note is as follows. 

In section 2, notation is introduced. Definitions of function spaces and results on norm equivalence from \cite{mf} are briefly mentioned at the convenience of readers. In section 3, we show continuity of the operators by using probability representations. In section 4, we derive some a priori estimates and prove the main theorem by applying continuation of parameters. Other auxiliary results are collected in the Appendix section.

\section{Notation and Function Spaces}
\subsection{Basic Notation}
We use $\mathbf{N}$ for the set of nonnegative integers, $\mathbf{N}_{+}$ for $\mathbf{N}\backslash\{0\}$, and $\Re$ for the real part of a complex-valued quantity.

For a function $u=u\left( t,x\right)$ on $H_T=[0,T]\times\mathbf{R}^d$, $\partial _{t}u:=\partial u/\partial t$, $\partial
_{i}u:=\partial u/\partial x_{i}$, $\partial _{ij}^{2}u:=\partial
^{2}u/\partial x_{i}x_{j}$. The gradient of $u$ with respect to $x$ is denoted by $\nabla u$, and $
D^{|\gamma |}u:=\partial ^{|\gamma |}u/\partial x_{1}^{\gamma _{1}}\ldots
\partial x_{d}^{\gamma _{d}}$, where $\gamma =\left( \gamma _{1},\ldots
,\gamma _{d}\right) \in \mathbf{N}^{d}$ is a multi-index.

As usual, $C_b^{\infty}\left( \mathbf{R}^d\right)$ denotes the set of infinitely differentiable functions on $\mathbf{R}^d$ whose derivative of arbitrary order is finite, $\mathcal{S}\left( \mathbf{R}^d\right)$ is the space of rapidly decreasing functions on $\mathbf{R}^d$ 
and $\mathcal{S}'\left( \mathbf{R}^d\right)$ denotes the space of continuous functionals on $\mathcal{S}\left( \mathbf{R}^d\right)$. It is well-known that Fourier transform is a bijection on $\mathcal{S}'\left( \mathbf{R}^d\right)$. We adopt the normalized definition for Fourier and its inverse transforms in this note, i.e.,
\begin{eqnarray*}
	\mathcal{F}\varphi\left(\xi\right) &=& \hat{\varphi}\left(\xi\right) :=\int e^{-i2\pi x\cdot \xi}\varphi\left(x\right)dx, \\
	\mathcal{F}^{-1}\varphi\left(x\right)&=&\check{\varphi}\left(x\right) := \int e^{i2\pi x\cdot \xi}\varphi\left(\xi\right)d\xi, \enskip \varphi\in\mathcal{S}\left( \mathbf{R}^d\right).
\end{eqnarray*}

For any L\'{e}vy measure $\nu$ we may symmetrize it as below.
\begin{eqnarray}
\bar{\nu}\left( dy\right):=\frac{1}{2}\left( \nu\left( dy\right)+\nu\left(- dy\right)\right).
\end{eqnarray}

As a convention, $C$ is a positive constant that represents different values in various contexts. Explicit dependence on certain quantities may be indicated when necessary.

\subsection{Function Spaces of Generalized Smoothness}

Our primary function spaces of generalized smoothness in this note are $\tilde{C}^{\beta}\left( \mathbf{R}^{d}\right),\beta \in \left( 0,1/\alpha\right)$ endowed with the norm
\begin{eqnarray*}
	\left\vert u\right\vert _{\beta }=\sup_{t,x}\left\vert u\left( t,x\right)
	\right\vert+\sup_{t,x,h\neq 0}\frac{\left\vert u\left(
		t,x+h\right) -u\left( t,x\right) \right\vert }{w\left( \left\vert h\right\vert\right)
		^{\beta }}:=\left\vert u\right\vert
	_{0}+\left[
u\right] _{\beta }<\infty  
\end{eqnarray*}
and $\tilde{C}^{1+\beta}\left( \mathbf{R}^{d}\right),\beta \in \left( 0,1/\alpha\right)$ with the norm
\begin{eqnarray*}
	\left\vert u\right\vert_{1+\beta }:=\left\vert u\right\vert
	_{0}+\left\vert L^{\mu}u\right\vert
	_{0}+\left[L^{\mu}
	u\right]_{\beta }<\infty,
\end{eqnarray*}
where $\mu$ is a reference measure satisfying \textbf{A(w,l)} for the same $w$ and $l$ as $\nu$, and $L^{\mu}$ is the associated operator defined as \eqref{op}. 

By \cite[Proposition 1]{mf}, these generalized H\"{o}lder norms are equivalent to the norm of generalized Besov spaces $\tilde{C}^{\beta}_{\infty,\infty}\left( \mathbf{R}^{d}\right)$:
\begin{eqnarray*}
	\left\vert u\right\vert_{\beta,\infty}=\sup_{j\in\mathbf{N}} w\left( N^{-j}\right)^{-\beta}\left\vert u\ast \varphi_j\right\vert _{0}<\infty, \quad \beta\in\left( 0,\infty\right).
\end{eqnarray*}
Given the choice of \cite{mf}, in above definition $\varphi_j\in\mathcal{S}\left( \mathbf{R}^{d}\right)$ for any $j\in\mathbf{N}$ and $\sum_{j=0}^{\infty}\mathcal{F}\varphi_j=1$. Moreover, when $j\geq 1$, $\mathcal{F}\varphi_j=\phi\left( N^{-j}\cdot\right)$ for some $N$ such that $l\left( N^{-1}\right)<1<l\left( N\right)$ and for some $\phi\in C_0^{\infty}\left(\mathbf{R}^d\right)$ so that $supp\left(\phi\right)=\{ \xi:N^{-1}\leq \left\vert \xi\right\vert\leq N\}$.

Set $\kappa\in\left[ 0,1\right]$ and $\beta>0$. Denote the L\'{e}vy symbol associated with $L^{\mu}$ by 
\begin{eqnarray*}
	\psi^{\mu}\left(\xi\right)=\int \left[ e^{i2\pi \xi\cdot y}-1-i2\pi\chi_{\alpha}\left( y\right)\xi\cdot y\right]\mu\left( dy\right), \xi\in\mathbf{R}^d,
\end{eqnarray*}
and denote
\begin{eqnarray*}
\psi^{\mu,\kappa}=\left\{\begin{array}{ll}
	\psi^{\mu} & \mbox{ if } \kappa=1,\\
	-\left(-\Re\psi^{\mu}\right)^{\kappa}  & \mbox{ if } \kappa\in\left( 0,1\right),\\
	1 & \mbox{ if } \kappa=0.
	\end{array}\right.
\end{eqnarray*}
Then the auxiliary space $C^{\mu,\kappa,\beta}\left( \mathbf{R}^{d}\right) $ is a class of functions whose norm
\begin{eqnarray*}
	\left\vert u\right\vert _{\mu,\kappa,\beta }:=\left\vert u\right\vert _{0}+\left\vert L^{\mu,\kappa}u\right\vert_{\beta,\infty}<\infty,
\end{eqnarray*}
where 
\begin{eqnarray}\label{opp}
	L^{\mu,\kappa}u:=\mathcal{F}^{-1}\left[ \psi^{\mu,\kappa}\mathcal{F}u\right], u\in \mathcal{S}'\left( \mathbf{R}^d\right).
\end{eqnarray}

Set
\begin{eqnarray*}
	\left( I-L^{\mu}\right)^{\kappa}u=\left\{\begin{array}{ll}
		\left( I-L^{\mu}\right)u & \mbox{ if } \kappa=1,\\
		\mathcal{F}^{-1}\left[ \left(1-\Re\psi^{\mu}\right)^{\kappa}\mathcal{F}u\right]  & \mbox{ if } \kappa\in\left[ 0,1\right).
	\end{array}\right.
\end{eqnarray*}
Another auxiliary space $\tilde{C}^{\mu,\kappa,\beta}\left( \mathbf{R}^{d}\right) $ is introduced as the collection of functions whose norm
\begin{eqnarray*}
	\left\Vert u\right\Vert _{\mu,\kappa,\beta }:=\left\vert \left( I-L^{\mu}\right)^{\kappa}u\right\vert_{\beta,\infty}<\infty.
\end{eqnarray*}

Lemmas \ref{rep}-\ref{rep2} below comprise a list of probabilistic representations that were derived in \cite{mf} and will be intensively used in next section. 
\begin{lemma}\cite[Lemma 8]{mf}\label{rep}
	Let $\nu$ be a L\'{e}vy measure satisfying (iii) in \textbf{A(w,l)} and $L^{\tilde{\nu}_R,\kappa}$ be defined as \eqref{opp}. Then for any $\varphi\left(x\right)\in C^{\infty}_b\left(\mathbf{R}^d\right)$, 
	\begin{eqnarray*}
	L^{\tilde{\nu}_R,\kappa}\varphi\left( x\right)=C\int_0^{\infty}t^{-1-\kappa}\mathbf{E}\left[\varphi\left( x+Z_t^{\overline{\tilde{\nu}_R}}\right)-\varphi\left( x\right)\right]dt,~\kappa\in\left( 0,1\right),\label{kap}
	\end{eqnarray*}
	where $C^{-1}=\int_0^{\infty}t^{-\kappa-1}\left(1- e^{-t}\right)dt$ and 
	\begin{eqnarray*}
		\overline{\tilde{\nu}_R}\left( dy\right)=\frac{1}{2}\left( \tilde{\nu}_R\left( dy\right)+\tilde{\nu}_R\left(- dy\right)\right),R>0.
	\end{eqnarray*}
	Besides, $L^{\tilde{\nu}_R,\kappa}\varphi\in C^{\infty}_b\left(\mathbf{R}^d\right)$. If furthermore $\varphi\left(x\right)\in \mathcal{S}\left(\mathbf{R}^d\right)$, then $\left\vert L^{\tilde{\nu}_R,\kappa}\varphi\right\vert_{L^1}<C'$ for some $C'>0$ uniform w.r.t. $R$.
\end{lemma}

\begin{lemma}\cite[Lemma 9]{mf}\label{bij}
	Let $a>0$ and $\nu$ be a L\'{e}vy measure satisfying (iii) in \textbf{A(w,l)}. Then $aI-L^{\nu}$ defines a bijection on $C_b^{\infty}\left( \mathbf{R}^d\right)$. Moreover, for all $C_b^{\infty}\left( \mathbf{R}^d\right)$ functions $\varphi$, the following representations hold.
	\begin{eqnarray*}
	\varphi\left( x\right) &=&\int_{0}^{\infty }e^{-a t}\mathbf{E}\left( aI
	-L^{\nu}\right) \varphi\left( x+Z_{t}^{\nu}\right) dt,\label{e1} \\
	\left( aI
	-L^{\nu}\right)^{-1}\varphi\left( x\right) &=&\int_{0}^{\infty }e^{-a t}\mathbf{E} \varphi\left( x+Z_{t}^{\nu}\right) dt,\quad x\in \mathbf{R}^{d}.\label{e2}
	\end{eqnarray*}
\end{lemma}

\begin{lemma}\cite[Lemma 10]{mf}\label{rep2}
	Let $a>0$ and $\kappa\in\left( 0,1\right)$. Suppose $\nu$ is a L\'{e}vy measure satisfying (iii) in \textbf{A(w,l)}.
	Then $\left( aI-L^{\nu}\right)^{\kappa}$ is a bijection on $C_b^{\infty}\left( \mathbf{R}^d\right)$. Moreover, for all $C_b^{\infty}\left( \mathbf{R}^d\right)$ functions $\varphi$, 
	\begin{eqnarray}
	\qquad\left( aI-L^{\nu}\right)^{\kappa}\varphi\left(x\right)&=&C\int_0^{\infty}t^{-\kappa-1}\left[ \varphi\left(x\right)-e^{-at}\mathbf{E}\varphi\left( x+Z_t^{\bar{\nu}}\right)\right]dt,\quad\label{rp1}\\
	\qquad\quad\left( aI-L^{\nu}\right)^{-\kappa}\varphi\left(x\right)&=&C'\int_0^{\infty}t^{\kappa-1}e^{-at}\mathbf{E}\varphi\left( x+Z_t^{\bar{\nu}}\right)dt,\label{rp2}
	\end{eqnarray}	
	where $C^{-1}=\int_0^{\infty}t^{-\kappa-1}\left(1- e^{-t}\right)dt$, $C'^{-1}=\int_0^{\infty}t^{\kappa-1}e^{-t}dt$ and $Z_t^{\bar{\nu}}$ is the L\'{e}vy process associated with $\bar{\nu}$.
\end{lemma}

\noindent\textbf{Remark:} Lemmas \ref{rep} and \ref{rep2} imply that $L^{\mu,\kappa}$, $\left( aI-L^{\nu}\right)^{\kappa}$, $\left( aI-L^{\nu}\right)^{-\kappa},\kappa\in\left( 0,1\right]$ are closed operations in $C_b^{\infty}\left( \mathbf{R}^d\right)$. Therefore, they may be all extended to $\kappa\in\left(1,2\right)$ through composition of operators. It was shown in \cite[Corollary 2]{mf} that \eqref{rp2} also holds for $ \kappa\in\left(1,2\right)$.

\begin{lemma}\cite[Lemma 6 and Proposition 6]{mf}\label{cont}
	Let $\beta>0$ and $\kappa\in\left[0,2\right)$. Suppose $\nu$ is a L\'{e}vy measure satisfying \textbf{A(w,l)}. Then \eqref{opp} is well-defined for all $\kappa$ and all $u\in\tilde{C}^{\kappa+\beta}_{\infty,\infty}\left(\mathbf{R}^d\right)$,
	\begin{eqnarray}\label{ext}
	L^{\nu,\kappa}u\left(x\right)=\lim_{n\rightarrow \infty}L^{\nu,\kappa}u_n\left(x\right),x\in\mathbf{R}^d,
	\end{eqnarray}
	and this convergence is uniform with respect to $x$. Moreover, 
	\begin{eqnarray*}
		\left\vert L^{\nu,\kappa}u\right\vert_{0}\leq\left\vert L^{\nu,\kappa}u\right\vert_{\beta,\infty}&\leq& C\left\vert u\right\vert_{\kappa+\beta,\infty}
	\end{eqnarray*}
	for some $C>0$ independent of $u$.
\end{lemma}

Based on Lemmas \ref{rep}-\ref{cont}, norm equivalence were established.
\begin{proposition}\cite[Theorems 3.2 and 3.3]{mf}\label{thm4}
	Let $\nu$ be a L\'{e}vy measure satisfying \textbf{A(w,l)}, $\beta>0,\kappa\in\left( 0,1\right]$. Then norms $\left\vert u\right\vert _{\nu,\kappa,\beta }$, $\left\Vert u\right\Vert _{\nu,\kappa,\beta }$ and $\left\vert u\right\vert_{\kappa+\beta,\infty}$ are mutually equivalent.	
\end{proposition}

\section{Continuity of the Operator}

In this section, we study respectively operators that have a kernel depending on the spatial variable $x$ and operators that have space-dependent coefficients. The first lemma explains the relation between generalized regularity and the ordinary smoothness.
\begin{lemma}\label{deri}
	Let $\beta,\delta\in\left(0,\infty\right)$, $\sigma\in\left[0,1\right)$ and $k$ be a positive integer so that%
	\begin{equation*}
	\int_{0}^{1}l\left( t\right) ^{\beta }t^{-k-1}dt<\infty .
	\end{equation*}%

	a) Any function $u\in \tilde{C}_{\infty ,\infty }^{\beta }\left( \mathbf{R}^{d}\right) $ is $k$-times continuously differentiable and there is $C$ depending only on $N,\beta$ so that for any multi-index $\left\vert \gamma \right\vert \leq k$ and any $\sigma\in\left[0,1\right)$ with $\left\vert\gamma\right\vert+\sigma\leq k$,
	\begin{equation*}
	\left\vert \partial^{\sigma}D^{\gamma }u\right\vert _{0}\leq C\left\vert u\right\vert _{\beta,\infty }\int_{0}^{1}l\left( t\right) ^{\beta }t^{-\left\vert \gamma\right\vert-\sigma-1}dt.
	\end{equation*}%
	Moreover,%
	\begin{equation*}
	\partial^{\sigma}D^{\gamma }u=D^{\gamma }\partial^{\sigma}u=\sum_{j=0}^{\infty }\left( \partial^{\sigma}D^{\gamma }u\right) \ast \varphi _{j}
	\end{equation*}%
	converges uniformly.
	
	b) Any function $u\in \tilde{C}_{\infty ,\infty }^{\beta +\delta }\left( 
	\mathbf{R}^{d}\right) $ is $k$-times continuously differentiable and there
	is $C$ depending only on $N,\beta$ so that for any multi-index $\left\vert \gamma \right\vert \leq k$ and any $\sigma\in\left[0,1\right)$ with $\left\vert\gamma\right\vert+\sigma\leq k$,
	\begin{equation*}
	\left\vert \partial^{\sigma}D^{\gamma }u\right\vert _{\delta,\infty }\leq C\left\vert u\right\vert
	_{\beta +\delta,\infty }\int_{0}^{1}l\left( t\right) ^{\beta }t^{-\left\vert \gamma\right\vert-\sigma-1}dt.
	\end{equation*}
\end{lemma}

\begin{proof}
	Recall properties of the convolution functions $\varphi_j, j\in\mathbf{N}$. If we write
	\begin{eqnarray*}
		&& \tilde{\varphi_j}=\varphi_{j-1}+\varphi_{j}+\varphi_{j+1}, j\geq 2,\label{ssch}\\
		&& \tilde{\varphi}_1=\check{\phi}+\varphi_{1}+\varphi_{2},\quad \tilde{\varphi}_0=\varphi_{0}+\varphi_{1},\label{sch}
	\end{eqnarray*}
	then, 
	\begin{eqnarray*}\label{schw}
		\mathcal{F}\tilde{\varphi}_j\left( \xi\right)=\hat{\tilde{\varphi}}_j\left( \xi\right)=\mathcal{F}\tilde{\varphi}\left( N^{-j}\xi\right), \quad \xi\in\mathbf{R}^d, j\geq 1,
	\end{eqnarray*}
	where
	\begin{eqnarray*}\label{schwa}
		\mathcal{F}\tilde{\varphi}\left( \xi\right)=\phi\left( N\xi\right)+\phi\left( \xi\right)+\phi\left( N^{-1}\xi\right).
	\end{eqnarray*}
	Note that $\phi$ is necessarily $0$ on the boundary of its support. Then,
	\begin{eqnarray*}\label{convol}
		\varphi_j &=& \varphi_j\ast\tilde{\varphi}_j, j\geq 0,\\
		\tilde{\varphi}_j\left( x\right) &=& N^{jd}\tilde{\varphi}\left( N^{j}x\right), j\geq 1.
	\end{eqnarray*}
	And then,
	\begin{eqnarray*}
		u=\sum_{j=0}^{\infty }u\ast \varphi _{j}=\sum_{j=0}^{\infty }\tilde{\varphi}_{j}\ast
		u\ast \varphi _{j}.
	\end{eqnarray*}
	
	a) We only show cases in which $\left\vert \gamma \right\vert =1$. The proof for other higher orders is an application of induction on $\gamma$. Denote%
	\begin{eqnarray*}
		\left( \partial^{\sigma}D^{\gamma }\tilde{\varphi}\right)_{j}\left( x\right) &=& N^{jd}\left(\partial^{\sigma}
		D^{\gamma }\tilde{\varphi}\right) \left( N^{j}x\right) ,x\in \mathbf{R}^{d}, j\geq 1.
	\end{eqnarray*}%
	Then%
	\begin{equation*}
	\sum_{j=1}^{\infty }\partial^{\sigma}D^{\gamma }\left(\tilde{\varphi}_{j}\ast
	u\ast \varphi _{j}\right) =\sum_{j=1}^{\infty }N^{\left(\left\vert \gamma \right\vert+\sigma\right) j}\left(\partial^{\sigma} D^{\gamma }%
	\tilde{\varphi}\right) _{j}\ast u\ast \varphi _{j}.
	\end{equation*}%
	
	Since
	\begin{equation*}
	\sum_{j=0}^{\infty }\frac{w\left( N^{-j}\right) ^{\beta }}{\left(
		N^{-j}\right)^{\left\vert \gamma \right\vert+\sigma}}\leq w\left( N\right)^{\beta}\int_{0}^{\infty }\frac{l\left( N^{-x}\right)
		^{\beta }}{\left( N^{-x}\right)^{\left\vert \gamma \right\vert+\sigma}}dx\leq C\int_{0}^{1}\frac{l\left( t\right)
		^{\beta }}{t^{\left\vert \gamma \right\vert+\sigma}}\frac{dt}{t}<\infty,
	\end{equation*}
	we have
	\begin{eqnarray*}
		\sum_{j=0}^{\infty }\left\vert \partial^{\sigma}D^{\gamma }\left(\tilde{\varphi}_{j}\ast
		u\ast \varphi _{j}\right) \right\vert_{0}
		&\leq& C\sum_{j=0}^{\infty }w\left( N^{-j}\right) ^{\beta }N^{\left(\left\vert \gamma \right\vert+\sigma\right) j}w\left( N^{-j}\right) ^{-\beta }\left\vert u\ast\varphi_j\right\vert_{0}\\
		&\leq& C\left\vert u\right\vert_{\beta,\infty }\int_{0}^{1}l\left( t\right) ^{\beta }t^{-\left\vert \gamma\right\vert-\sigma-1}dt<\infty.
	\end{eqnarray*}
	Therefore, $\sum_{j=0}^{\infty }\partial^{\sigma}D^{\gamma }\left(\tilde{\varphi}_{j}\ast
	u\ast \varphi _{j}\right)\in C\left( \mathbf{R}^d\right)$ converges uniformly and therefore it converges in the weak topology of $\mathcal{S}'\left( \mathbf{R}\right)$. By continuity of the Fourier transform,
	\begin{equation*}
	D^{\gamma }\partial^{\sigma}u=\partial^{\sigma}D^{\gamma }u=\sum_{j=0}^{\infty }\partial^{\sigma}D^{\gamma }\left(\tilde{\varphi}_{j}\ast
	u\ast \varphi _{j}\right) =\sum_{j=0}^{\infty }\left( \partial^{\sigma}D^{\gamma }u\right)\ast \varphi _{j}.
	\end{equation*}
	Moreover,
	\begin{eqnarray*}
		\left\vert \partial^{\sigma}D^{\gamma }u\right\vert_{0}\leq \sum_{j=0}^{\infty }\left\vert \partial^{\sigma}D^{\gamma }\left(\tilde{\varphi}_{j}\ast
		u\ast \varphi _{j}\right) \right\vert_{0}\leq C\left\vert u\right\vert_{\beta,\infty }\int_{0}^{1}l\left( t\right) ^{\beta }t^{-\left\vert \gamma\right\vert-\sigma-1}dt.
	\end{eqnarray*}
	b) From a),
	\begin{eqnarray*}
		&&w\left( N^{-j}\right) ^{-\delta }\left\vert\left( \partial^{\sigma}D^{\gamma}
		u\right)\ast \varphi _{j} \right\vert_{0}\\
		&\leq& C\sum_{j=0}^{\infty }w\left( N^{-j}\right) ^{\beta }N^{\left(\left\vert \gamma \right\vert+\sigma\right) j}w\left( N^{-j}\right) ^{-\beta-\delta }\left\vert u\ast\varphi_j\right\vert_{0}\\
		&\leq& C\left\vert u\right\vert_{\beta+\delta,\infty }\int_{0}^{1}l\left( t\right) ^{\beta }t^{-\left\vert \gamma\right\vert-\sigma-1}dt, \enskip \forall j\in\mathbf{N}.
	\end{eqnarray*}
	And the conclusion follows.
\end{proof}
\noindent\textbf{Remark:} As a conclusion of Lemma \ref{deri} and \textbf{$\mathbf{\tilde{A}}(w,l,\gamma)$} (iii), if $u\in\tilde{C}_{\infty ,\infty }^{1+\beta }\left(\mathbf{R}^d\right),\beta> 0$ and $\alpha\in\left[1,2\right)$, then $u$ has classical first-order derivatives.

\begin{lemma}\label{ato}
	Let $\kappa\in\left(0,2\right)$ and $\mu$ be the reference measure. Then for any function $\varphi\in C^{\infty}_b\left( \mathbf{R}^d\right)$,
	\begin{eqnarray*}
		\left( aI-L^{\mu}\right)^{\kappa}\varphi&\rightarrow& -L^{\mu,\kappa}\varphi, \kappa\in\left(0,1\right] ,\\
		\left( aI-L^{\mu}\right)^{\kappa}\varphi&\rightarrow& L^{\mu,\kappa}\varphi, \kappa\in\left(1,2\right) .
	\end{eqnarray*}
	uniformly as $a\rightarrow 0_{+}$.
\end{lemma}
\begin{proof}
	Apparently, $\left( aI-L^{\mu}\right)\varphi\left( x\right)\rightarrow -L^{\mu}\varphi\left( x\right)$ uniformly as $a\rightarrow 0$. Use the representation $\eqref{rp1}$ for $\kappa\in\left(0,1\right)$:
	\begin{eqnarray*}
		\left( aI-L^{\mu}\right)^{\kappa}\varphi\left(x\right)&=&C\left(\kappa\right)\int_0^{\infty}t^{-\kappa-1}e^{-at}\left[ \varphi\left(x\right)-\mathbf{E}\varphi\left( x+Z_t^{\bar{\mu}}\right)\right]dt+a^{\kappa}\varphi\left(x\right),
	\end{eqnarray*}	
	where $C\left(\kappa\right)^{-1}=\int_0^{\infty}t^{-\kappa-1}\left(1- e^{-t}\right)dt$. Note $\eqref{kap}$, then
	\begin{eqnarray*}
		&&\left\vert\left( aI-L^{\mu}\right)^{\kappa}\varphi\left(x\right)+L^{\mu,\kappa}\varphi\left(x\right)\right\vert\\
		&\leq&C\left(\kappa\right)\int_0^{\infty}t^{-\kappa-1}\left\vert e^{-at}-1\right\vert \left\vert \varphi\left(x\right)-\mathbf{E}\varphi\left( x+Z_t^{\bar{\mu}}\right)\right\vert dt+a^{\kappa}\left\vert \varphi\left(x\right)\right\vert\\
		&\leq& 2C\left(\kappa\right)\left\vert \varphi\right\vert_0 \left[ a^{\kappa}\int_0^{a}t^{-\kappa-1}\left(1-e^{-t}\right)dt +\int_1^{\infty}t^{-\kappa-1}\left(1-e^{-at}\right)dt+a^{\kappa}\right] \\
		&\to& 0 \mbox{ uniformly as } a\to 0_{+}, \forall x\in\mathbf{R}^d.
	\end{eqnarray*}	
	 To be precise, for any $\varepsilon>0$, there is $\delta>0$ such that
	 \begin{eqnarray*}
	 \left\vert\left( aI-L^{\mu}\right)^{\kappa}\varphi+L^{\mu,\kappa}\varphi\right\vert_0<\varepsilon	\left\vert \varphi\right\vert_0 
	 \end{eqnarray*}
	 whenever $0<a<\delta$. Besides, 
	\begin{eqnarray*}
		&&\left( aI-L^{\mu}\right)^{2\kappa}\varphi-L^{\mu,2\kappa}\varphi\\
		&=& \left[\left( aI-L^{\mu}\right)^{\kappa}+L^{\mu,\kappa}\right]\circ \left[\left( aI-L^{\mu}\right)^{\kappa}+L^{\mu,\kappa}\right]\circ \varphi\\
		&-&2\left( aI-L^{\mu}\right)^{\kappa}\circ L^{\mu,\kappa}\varphi-2L^{\mu,2\kappa}\varphi.
	\end{eqnarray*}

	By arguments above, when $0<a<\delta$,
	\begin{eqnarray*}
	\left\vert \left[\left( aI-L^{\mu}\right)^{\kappa}+L^{\mu,\kappa}\right]\circ \left[\left( aI-L^{\mu}\right)^{\kappa}+L^{\mu,\kappa}\right]\circ \varphi\right\vert_0\leq\varepsilon^2\left\vert \varphi\right\vert_0,
	\end{eqnarray*}
and 
	\begin{eqnarray*}
		-2\left( aI-L^{\mu}\right)^{\kappa}\circ L^{\mu,\kappa}\varphi-2L^{\mu,2\kappa}\varphi&\to& 0 \mbox{ uniformly as } a\to 0_{+}.
	\end{eqnarray*}
	Therefore,
		\begin{eqnarray*}
		\left( aI-L^{\mu}\right)^{2\kappa}\varphi&\to& L^{\mu,2\kappa}\varphi \mbox{ uniformly as } a\to 0_{+}.
	\end{eqnarray*}
\end{proof}

The following derivation is needed in next two lemmas. Given $\eqref{alpha1}$,  
\begin{equation*}
\psi^{\mu}\left( \xi\right) = w\left( R\right)^{-1}\psi^{\tilde{\mu}_{R}}\left( R\xi\right), \xi\in\mathbf{R}^d, \forall R\in\mathbf{R}_{+}.
\end{equation*}
Using the L\'{e}vy-Khintchine formula, we obtain
\begin{equation*}
p\left( t,z\right) = R^{-d} p^R \left( w\left( R\right)^{-1}t, R^{-1}z\right), z\in\mathbf{R}^d, \forall R\in\mathbf{R}_{+},
\end{equation*}
where $p\left( t,z\right), z\in\mathbf{R}^d$ denotes the density function of $Z_{t}^{\mu}$ if $\kappa=1$ and that of $Z_{t}^{\bar{\mu}}$ if $\kappa\in\left( 0,1\right)\cup\left(1,2\right)$, $p^R\left( t,z\right), z\in\mathbf{R}^d$ denotes the density of $Z_{t}^R:=Z_{t}^{\tilde{\mu}_R}$ if $\kappa=1$ and $Z_{t}^{\tilde{\bar{\mu}}_R}$ otherwise. Existence of $p^R\left( t,z\right)$ is guaranteed by Lemma \ref{lemma3} in Appendix.

\begin{lemma}\label{dif}
	Let $\kappa\in\left(0,2\right),\beta\in\left( 0,\infty\right)$ and $\mu$ be the reference measure. Assume
	\begin{eqnarray}\label{ass1}
	\int_{1}^{\infty }t^{\kappa-1}\gamma\left( t\right)^{-1}dt<\infty.
	\end{eqnarray}
	Then for any function $\varphi\in \tilde{C}^{\kappa+\beta}_{\infty,\infty}\left( \mathbf{R}^d\right)$ and any $R>0$,
	\begin{eqnarray}
\varphi\left( x+y\right)-\varphi\left( x\right)
	&=&C\left(\kappa \right)w\left( R\right)^{\kappa}\int_{0}^{\infty }t^{\kappa-1}\int L^{\mu,\kappa}\varphi\left( x+R z\right)\nonumber\\
	&&\cdot\left[ p^{R}\left( t,z-R^{-1}y\right)-p^{R}\left( t,z\right)\right]dzdt,\quad\quad\label{tran21}
	\end{eqnarray}
	where $p^{R}\left( t,x\right), x\in\mathbf{R}^d$ follows the definition above. In particular,
	\begin{eqnarray}
	\left\vert\varphi\left( x+y\right)-\varphi\left( x\right) \right\vert\leq C w\left( \left\vert y\right\vert\right)^{\kappa}\left\vert L^{\mu,\kappa}\varphi\right\vert_0,\forall x,y\in\mathbf{R}^d. \label{mod1}
	\end{eqnarray}
\end{lemma}
\begin{proof}
	We first assume $\varphi\in C_b^{\infty}\left( \mathbf{R}^d\right)\cap \tilde{C}^{\kappa+\beta}_{\infty,\infty}\left( \mathbf{R}^d\right)$. By \eqref{rp2},
	\begin{eqnarray*}
	&&\varphi\left( x+y\right)-\varphi\left( x\right)\nonumber\\
	&=&C\int_{0}^{\infty }t^{\kappa-1}e^{-a t}\mathbf{E}\left[ 
	\left( aI-L^{\mu}\right)^{\kappa}\varphi\left( x+y+Z_{t}\right)-\left( aI-L^{\mu}\right)^{\kappa}\varphi\left( x+Z_{t}\right)\right] dt \nonumber\\
    &=&C\int_{0}^{\infty }t^{\kappa-1}e^{-a t}\int\left( aI-L^{\mu}\right)^{\kappa}\varphi\left( x+z\right)\left[ 
	p\left( t,z-y\right)-p\left(t, z\right)\right] dz dt, 
	\end{eqnarray*}
	where $Z_t=Z_{t}^{\mu}$ if $\kappa=1$ and $Z_t=Z_{t}^{\bar{\mu}}$ otherwise, and $p\left( t,x\right)$ denotes the probability density function of $Z_t$. Recall that Lemma \ref{lemma3} claims
	\begin{eqnarray*}
		\int \left\vert \nabla p\left( t,z\right)\right\vert dz<C'\gamma\left( t\right)^{-1}.
	\end{eqnarray*}
	Let $a\rightarrow 0$ under $\eqref{ass1}$. By Lemma \ref{ato}, for all $\kappa\in\left(0,2\right)$,
	\begin{eqnarray}
	&&\varphi\left( x+y\right)-\varphi\left( x\right)\nonumber\\ &=&C\left( \kappa\right)\int_{0}^{\infty }t^{\kappa-1}\int L^{\mu,\kappa}\varphi\left( x+z\right)\left[ p\left( t,z-y\right)-p\left( t,z\right)\right]dzdt\quad\nonumber\\
	&=&C\left( \kappa\right)w\left( R\right)^{\kappa}\int_{0}^{\infty }t^{\kappa-1}\int L^{\mu,\kappa}\varphi\left( x+Rz\right)\left[ p^R\left( t,z-R^{-1}y\right)-p^R\left( t,z\right)\right]dzdt.\qquad\nonumber
	\end{eqnarray}
	
	Now consider $\varphi\in  \tilde{C}^{\kappa+\beta}_{\infty,\infty}\left( \mathbf{R}^d\right)$. By \cite[Proposition 5]{mf} and Lemma \ref{cont}, there is a sequence of functions $v_n\in C_b^{\infty}\left( \mathbf{R}^d\right)\cap \tilde{C}^{\kappa+\beta}_{\infty,\infty}\left( \mathbf{R}^d\right)$ such that 
	\begin{eqnarray*}
		\lim_{n\rightarrow \infty}\left\vert L^{\mu,\kappa}v_n-L^{\mu,\kappa}\varphi\right\vert_0=0,\forall \kappa\in\left[0,2\right).
	\end{eqnarray*}
Moreover,
\begin{eqnarray*}
v_n\left( x+y\right)-v_n\left( x\right)
&=&C\left( \kappa\right)w\left( R\right)^{\kappa}\int_{0}^{\infty }t^{\kappa-1}\int L^{\mu,\kappa}v_n\left( x+Rz\right)\\
&&\cdot \left[ p^R\left( t,z-R^{-1}y\right)-p^R\left( t,z\right)\right]dzdt.
\end{eqnarray*}
	Pass the limit on both sides. Then $\eqref{tran21}$ holds for $ \varphi\in\tilde{C}^{\kappa+\beta}_{\infty,\infty}\left( \mathbf{R}^d\right)$. If $y\neq 0$, by setting $R=\left\vert y\right\vert$, we obtain $\eqref{mod1}$ immediately.
\end{proof}

Denote
\begin{eqnarray*}
	\nabla^{\alpha}u\left( x;y\right)&=&u\left( x+y\right)-u\left( x\right)-\chi_{\alpha}\left( y\right)y\cdot \nabla u\left( x\right).
\end{eqnarray*}
\begin{lemma}\label{diff}
	Let $\kappa\in\left(0,2\right), \beta\in\left(0,\infty\right)$ and $\mu$ be the reference measure. Assume
	\begin{eqnarray}
	1_{\alpha\in\left[1,2\right) }\int_0^{1}t^{\kappa-1}\gamma\left( t\right)^{-1}dt&<&\infty,\qquad\qquad\label{ass21}\\
	 \qquad \int_1^{\infty}1_{\alpha\in \left(0,1\right)}t^{\kappa-1}\gamma\left( t\right)^{-1}+ 1_{\alpha\in\left[1,2\right)}t^{\kappa-1}\gamma\left( t\right)^{-2}dt&<&\infty.\label{ass22}
	\end{eqnarray}
	Then for all $\alpha\in\left( 0,1\right)\cup\left( 1,2\right)$ and any function $\varphi\in \tilde{C}^{\kappa+\beta}_{\infty,\infty}\left( \mathbf{R}^d\right)$,
	\begin{eqnarray}
\nabla^{\alpha}\varphi\left( x;y\right)
	&=&C\left(\kappa \right)\int_{0}^{\infty }t^{\kappa-1}\int L^{\mu,\kappa}\varphi\left( x+z\right)\nabla^{\alpha} p\left( t,z;-y\right)dzdt\nonumber\\
	&=&C\left(\kappa \right)w\left( R\right)^{\kappa}\int_{0}^{\infty }t^{\kappa-1}\int L^{\mu,\kappa}\varphi\left( x+R z\right)\nonumber\\
	&&\nabla^{\alpha} p^{R}\left( t,z;-R^{-1}y\right)dzdt, \forall R>0.\qquad\label{tran2}
	\end{eqnarray}
	Moreover,
	\begin{eqnarray}
	\left\vert\nabla^{\alpha}\varphi\left( x;y\right)\right\vert\leq C w\left( \left\vert y\right\vert\right)^{\kappa}\left\vert L^{\mu,\kappa}\varphi\right\vert_0, \forall x,y\in\mathbf{R}^d.\label{mod2}
	\end{eqnarray}

	If $\alpha=1$, then $\eqref{mod2}$ hold for $\left\vert y\right\vert \leq 1$ and all $\varphi\in \tilde{C}^{\kappa+\beta}_{\infty,\infty}\left( \mathbf{R}^d\right)$.
\end{lemma}
\begin{proof}
	Let $a>0$. Similarly as in Lemma \ref{dif}, we first consider $\varphi\in C_b^{\infty}\left( \mathbf{R}^d\right)\cap\tilde{C}^{\kappa+\beta}_{\infty,\infty}\left( \mathbf{R}^d\right)$. By $\eqref{rp2}$,
	\begin{eqnarray*}
		\nabla^{\alpha}\varphi\left( x;y\right)=C\int_0^{\infty}t^{\kappa-1}e^{-at}\int\left( aI-L^{\mu}\right)^{\kappa}\varphi\left(x+z\right)\nabla^{\alpha}p \left( t,z;-y\right)dzdt,
	\end{eqnarray*}
	where $C=\left( \int_0^{\infty}t^{\kappa-1}e^{-t}dt\right)^{-1}$. Let $a\rightarrow 0$ under assumptions $\eqref{ass21}, \eqref{ass22}$. Then for all $\kappa\in\left(0,2\right)$,
	\begin{eqnarray*}
		\nabla^{\alpha}\varphi\left( x;y\right)
		&=&C\int_0^{\infty}t^{\kappa-1}\int L^{\mu,\kappa}\varphi\left(x+z\right)\nabla^{\alpha}p \left( t,z;-y\right)dzdt\\
		&=& Cw\left( R\right)^{\kappa}\int_{0}^{\infty }t^{\kappa-1}\int L^{\mu,\kappa}\varphi\left( x+Rz\right)\nabla^{\alpha} p^R\left( t,z;-R^{-1}y\right)dzdt.
	\end{eqnarray*}
	
	For general $\varphi\in  \tilde{C}^{\kappa+\beta}_{\infty,\infty}\left( \mathbf{R}^d\right)$. By \cite[Proposition 5]{mf} and Lemma \ref{cont}, there is a sequence of functions $v_n\in C_b^{\infty}\left( \mathbf{R}^d\right)\cap \tilde{C}^{\kappa+\beta}_{\infty,\infty}\left( \mathbf{R}^d\right)$ such that 
	\begin{eqnarray*}
		\lim_{n\rightarrow \infty}\left\vert L^{\mu,\kappa}v_n-L^{\mu,\kappa}\varphi\right\vert_0=0,\forall \kappa\in\left[0,2\right),
	\end{eqnarray*}
and
\begin{eqnarray*}
	\nabla^{\alpha}v_n\left( x;y\right)
	&=& Cw\left( R\right)^{\kappa}\int_{0}^{\infty }t^{\kappa-1}\int L^{\mu,\kappa}v_n\left( x+Rz\right)\nabla^{\alpha} p^R\left( t,z;-R^{-1}y\right)dzdt.
\end{eqnarray*}
	Passing the limit on both sides, we obtain $\eqref{tran2}$ for all functions in $ \tilde{C}^{\kappa+\beta}_{\infty,\infty}\left( \mathbf{R}^d\right)$. Setting $R=\left\vert y\right\vert,y\neq 0$, we have
	\begin{eqnarray}
\nabla^{\alpha}\varphi\left( x;y\right)
	&=&	C\left(\kappa \right)w\left( \left\vert y\right\vert\right)^{\kappa}\int_{0}^{\infty }t^{\kappa-1}\int L^{\mu,\kappa}\varphi\left( x+\left\vert y\right\vert z\right)\notag\\
	&&\nabla^{\alpha} p^{\left\vert y\right\vert}\left( t,z;-\left\vert y\right\vert^{-1}y\right)dzdt.\qquad\label{ddd}
	\end{eqnarray}
	
	$\eqref{mod2}$ then follows from $\eqref{ddd},\eqref{ass21}$ and $\eqref{ass22}$.
\end{proof}

 Now we are ready to prove the stronger continuity of the operator. Choose $\eta\left(x\right)\in C_0^{\infty}\left( \mathbf{R}^d\right)$ such that $0\leq \eta\left( x\right)\leq 1,\forall x\in \mathbf{R}^d, supp \left( \eta\right)\subseteq \{x:\left\vert x\right\vert\leq 2 \}$, and $\eta\left( x\right)\equiv 1$ on $\overline{B_1\left( 0\right)}$. $\eta_{m,z}\left( x\right):=\eta\left( m\left( x-z\right)\right),m\geq 1$,

\subsection{Operators with Space-Dependent Kernels }
Let $\nu$ be a L\'{e}vy measure satisfying \textbf{$\mathbf{\tilde{A}}(w,l,\gamma)$}. We now consider $\rho\left( t,x, y\right)\nu\left( dy\right)$, where $\rho\left(t, x, y\right)$ satisfies \textbf{H($K,\beta$)}. Obviously, $\rho\left( t,x, y\right)\nu\left( dy\right)$ is a L\'{e}vy measure for each fixed $x\in\mathbf{R}^d$ and $t\in\left[0,T\right]$. Denote
\begin{eqnarray}
	&&L_{t,z} u\left(x\right)\label{opp3}\\
	&=&\int \left[ u\left( x+y\right)-u\left(t, x\right)-\chi_{\alpha}\left( y\right)y\cdot \nabla u\left( x\right)\right]\rho\left( t,z, y\right)\nu\left( dy\right),\notag\\
	&&\langle u,\eta_{m,z}\rangle_{t,z} \label{opp4}\\
	&=&\int\left[ u\left(  x+y\right)-u\left(  x\right)\right]\left[ \eta_{m,z}\left(  x+y\right)-\eta_{m,z}\left(  x\right)\right]\rho\left( t, z,y\right) \nu\left( dy\right).\notag
\end{eqnarray}

\begin{lemma}\label{cont2}
	Let $\nu$ be a L\'{e}vy measure satisfying \textbf{$\mathbf{\tilde{A}}(w,l,\gamma)$} and $\rho$ be a bounded measurable function. $\beta\in\left( 0,1/\alpha\right)$. $u\in \tilde{C}_{\infty ,\infty }^{1+\beta }\left(\mathbf{R}^d\right)$. Then there is $\beta'\in\left( 0,\beta\right)$ such that
	\begin{eqnarray*}
		\left\vert L_{t,z} u \right\vert_{0}&\leq& C\sup_{t,z,y}\left\vert \rho\left( t, z,y\right)\right\vert\left\vert u\right\vert_{1+\beta',\infty},\\
		\left[ L_{t,z} u\right]_{\beta}&\leq& C\sup_{t,z,y}\left\vert \rho\left( t, z,y\right)\right\vert\left\vert u\right\vert_{1+\beta,\infty}.
	\end{eqnarray*}
	where $C$ does not depend on $t,z$ or $u$.
\end{lemma}
\begin{proof}
	Recall parameters introduced in \textbf{A(w,l)}. Clearly,
	\begin{eqnarray*}
		\left\vert L_{t,z} u \right\vert_{0}\leq \sup_{t,z,y}\left\vert \rho\left( t, z,y\right)\right\vert\left(\int_{\left\vert y\right\vert\leq 1} \left\vert \nabla^{\alpha}u\left(x;y\right)\right\vert\nu\left( dy\right)+\int_{\left\vert y\right\vert> 1} \left\vert \nabla^{\alpha}u\left( x;y\right)\right\vert\nu\left( dy\right)\right).
	\end{eqnarray*}
	
	Choose $\kappa\in\left(0,1\right)$ sufficiently small so that $\lim_{r\rightarrow \infty}w\left( r\right)^{\kappa}/r^{\alpha_2}=0$. According to \cite[Lemma 1]{mf}, such a $\kappa$ must exist. Then by $\eqref{mod1}$ and \textbf{$\mathbf{\tilde{A}}(w,l,\gamma)$}(iii), for all $\alpha\in\left(0,1\right]$,
	\begin{eqnarray*}
		\int_{\left\vert y\right\vert> 1} \left\vert \nabla^{\alpha}u\left(x;y\right)\right\vert\nu\left( dy\right)
		&\leq& C\left\vert L^{\mu,\kappa}u\right\vert_0\int_{\left\vert y\right\vert> 1}  w\left( \left\vert y\right\vert\right)^{\kappa} \nu\left( dy\right)\\
		&\leq& C\left(d,\kappa,\alpha\right)\left\vert L^{\mu,\kappa}u\right\vert_0, \forall x\in\mathbf{R}^d.
	\end{eqnarray*}
	If $\alpha\in\left(1,2\right)$, then we use Lemma \ref{deri} and \textbf{A(w,l)}.
	\begin{eqnarray*}
		\int_{\left\vert y\right\vert> 1} \left\vert \nabla^{\alpha}u\left(x;y\right)\right\vert\nu\left( dy\right)
		&\leq& C\left(d,\alpha\right)\left( \left\vert u\right\vert_0+\left\vert u\right\vert_{1,\infty}\right), \forall x\in\mathbf{R}^d.
	\end{eqnarray*}
	It follows from \cite[Proposition 4]{mf} and Lemma \ref{cont} that for all $\alpha\in\left(0,2\right)$ and any $\beta'\in\left( 0,\beta\right)$, 
	\begin{eqnarray*}
		\int_{\left\vert y\right\vert> 1} \left\vert \nabla^{\alpha}u\left(x;y\right)\right\vert\nu\left( dy\right)
		\leq C\left(d,\kappa,\alpha,\beta'\right)\left\vert u\right\vert_{1+\beta',\infty},\forall x\in\mathbf{R}^d.
	\end{eqnarray*}
	
	On the other hand, suggested by \textbf{$\mathbf{\tilde{A}}(w,l,\gamma)$}(iii), we apply $\eqref{mod2}$ by setting $\beta'\in\left( 0,\delta\right)$ if $\alpha\neq 1$ and $\beta'=\delta$ if $\alpha= 1$. Then
	\begin{eqnarray*}
		&&\int_{\left\vert y\right\vert\leq 1} \left\vert \nabla^{\alpha}u\left(x;y\right)\right\vert\nu\left( dy\right)\\
		&\leq& C\left(d,\beta',\alpha\right)\left\vert L^{\mu,1+\beta'}u\right\vert_{0}\int_{\left\vert y\right\vert\leq 1}w\left( \left\vert y\right\vert\right)^{1+\beta'}\nu\left( dy\right),\forall x\in\mathbf{R}^d.
	\end{eqnarray*}
	By Lemma \ref{ess} (c) and Lemma \ref{cont},
	\begin{eqnarray*}
		\int_{\left\vert y\right\vert\leq 1} \left\vert \nabla^{\alpha}u\left(x;y\right)\right\vert\nu\left( dy\right)
		&\leq& C\left(d,\beta',\alpha\right)\left\vert u\right\vert_{1+\left(\beta+\beta'\right)/2,\infty},\forall x\in\mathbf{R}^d.
	\end{eqnarray*}
	
	Now consider $\left\vert L_z u\left(x_1\right)-
	L_z u\left(x_2\right)\right\vert$. Set $a=\left\vert x_1-x_2\right\vert$. Then,
	\begin{eqnarray*}
		&&\left\vert L_{t,z} u\left(x_1\right)-
		L_{t,z} u\left(x_2\right) \right\vert\\
		&\leq& \sup_{t,z,y}\left\vert \rho\left( t, z,y\right)\right\vert\int_{\left\vert y\right\vert\leq a} \left\vert \nabla^{\alpha}\left(u\left(x_1;y\right)-u\left(x_2;y\right)\right)\right\vert\nu\left( dy\right)\\
		&+&\sup_{t,z,y}\left\vert \rho\left( t, z,y\right)\right\vert\int_{\left\vert y\right\vert> a} \left\vert \nabla^{\alpha}\left(u\left(x_1;y\right)-u\left(x_2;y\right)\right)\right\vert\nu\left( dy\right).
	\end{eqnarray*}
	
	Denote $\varsigma\left( r\right)=\nu\left( \left\vert y\right\vert >r\right)$ and take $\beta'\in\left( 0,\delta\right)$ if $\alpha\neq 1$ and $\beta'=\delta$ if $\alpha= 1$. Apply $\eqref{ddd}$, \cite[Proposition 1]{mf}, Lemmas \ref{cont} and \ref{ess}(a).
	\begin{eqnarray*}
		&&\int_{\left\vert y\right\vert\leq a} \left\vert \nabla^{\alpha}\left(u\left(x_1;y\right)-u\left(x_2;y\right)\right)\right\vert\nu\left( dy\right)\\
		&\leq& C\sup_z\left\vert L^{\mu,1+\beta'}u\left(x_1+ z\right)-L^{\mu,1+\beta'}u\left(x_2+z\right)\right\vert\int_{\left\vert y\right\vert\leq a}w\left( \left\vert y\right\vert\right)^{1+\beta'} \nu\left( dy\right)\\
		&\leq& -C\left[L^{\mu,1+\beta'}u\right]_{\beta-\beta'}w\left( a\right)^{\beta-\beta'}\int_{0}^a \varsigma\left( r\right)^{-1-\beta'} d\varsigma\left( r\right)\\
		&\leq& C\left\vert u\right\vert_{1+\beta,\infty}w\left( a\right)^{\beta-\beta'} \varsigma\left( r\right)^{-\beta'}|_{0}^{a}\leq C\left\vert u\right\vert_{1+\beta,\infty} w\left( a\right)^{\beta}.
	\end{eqnarray*}
	
	Recall \textbf{$\mathbf{\tilde{A}}(w,l,\gamma)$} and set $\kappa=\beta+\min\left( \delta,\varepsilon\right)/2$ if $\alpha\neq 1$ and $\kappa=\beta+\left( \delta'+\varepsilon\right)/2$ if $\alpha=1$. Apply $\eqref{ddd}$, \cite[Proposition 1]{mf} and Proposition \ref{thm4}. 
	\begin{eqnarray*}
		&&\int_{\left\vert y\right\vert> a} \left\vert \nabla^{\alpha}\left(u\left(x_1;y\right)-u\left(x_2;y\right)\right)\right\vert\nu\left( dy\right)\\
		&\leq& C\sup_z\left\vert L^{\mu,1+\beta-\kappa}u\left(x_1+ z\right)-L^{\mu,1+\beta-\kappa}u\left(x_2+z\right)\right\vert\int_{\left\vert y\right\vert> a}w\left( \left\vert y\right\vert\right)^{1+\beta-\kappa}\nu\left( dy\right)\\
		&\leq& C\left\vert u\right\vert_{1+\beta,\infty}w\left( a\right)^{\kappa}\int_{\left\vert y\right\vert> a}w\left( \left\vert y\right\vert\right)^{1+\beta-\kappa}\nu\left( dy\right).
	\end{eqnarray*}
	Similarly as above,
	\begin{eqnarray*}
		&&\int_{\left\vert y\right\vert> a} \left\vert \nabla^{\alpha}\left(u\left(x_1;y\right)-u\left(x_2;y\right)\right)\right\vert\nu\left( dy\right)\\
		&\leq&- C\left\vert u\right\vert_{1+\beta,\infty}w\left( a\right)^{\kappa}\int_{ a}^{\infty}\varsigma\left( r\right)^{-1-\beta+\kappa}d\varsigma\left(r\right)\\
		&\leq& -C \left\vert u\right\vert_{1+\beta,\infty}w\left( a\right)^{\kappa}\varsigma\left(r\right)^{\kappa-\beta}|_{a}^{\infty}\leq  C \left\vert u\right\vert_{1+\beta,\infty}w\left( a\right)^{\beta}.
	\end{eqnarray*}

	As a conclusion, $\left[ L_{t,z} u\right]_{\beta}\leq C\sup_{t,z,y}\left\vert \rho\left( t, z,y\right)\right\vert\left\vert u\right\vert_{1+\beta,\infty}$.
\end{proof}

\begin{corollary}\label{coo}
	Let $\nu$ be a L\'{e}vy measure satisfying \textbf{$\mathbf{\tilde{A}}(w,l,\gamma)$} and $\rho$ satisfy \textbf{H($K,\beta$)}, $\beta\in\left( 0,1/\alpha\right)$. Then for any $u\in \tilde{C}_{\infty ,\infty }^{1+\beta }\left(\mathbf{R}^d\right)$,
	\begin{eqnarray*}
		\left\vert\mathcal{A} u \right\vert_{\beta}&\leq& C\left(\sup_{t,z,y}\left\vert \rho\left( t, z,y\right)\right\vert\left\vert u \right\vert_{1+\beta,\infty}+\sup_{t,y}\left\vert\rho\left( t, \cdot,y\right)\right\vert_{\beta}\left\vert u\right\vert_{1+\beta',\infty}\right),
	\end{eqnarray*}
	where $\beta'\in\left( 0,\beta\right)$ and $C$ does not depend on $u$.
\end{corollary}
\begin{proof}
	Obviously, $\left\vert \mathcal{A} u \right\vert_{0}\leq \sup_{t,z}\left\vert L_{t,z} u \right\vert_{0}\leq C\sup_{t,z,y}\left\vert \rho\left( t, z,y\right)\right\vert\left\vert  u \right\vert_{1+\beta',\infty}$ for some $\beta'\in\left( 0,\beta\right)$. Meanwhile,
	\begin{eqnarray*}
		&&\left\vert L_{t,x+y} u\left( x+y\right)-L_{t,x} u\left( x\right)\right\vert\\
		&\leq&\left\vert L_{t,x+y} u\left( x+y\right)-L_{t,x} u\left( x+y\right)\right\vert+\left\vert L_{t,x} u\left( x+y\right)-L_{t,x} u\left( x\right)\right\vert\\
		&\leq& C\sup_{t,y}[\rho\left( t, \cdot,y\right)]_{\beta}\left\vert u\right\vert_{1+\beta',\infty} w\left( \left\vert y\right\vert\right)^{\beta}+ C \sup_{t,z,y}\left\vert \rho\left( t, z,y\right)\right\vert\left\vert u\right\vert_{1+\beta,\infty} w\left( \left\vert y\right\vert\right)^{\beta}.
	\end{eqnarray*} 
	Namely, $\left[\mathcal{A} u\right]_{\beta}\leq C\left(\sup_{t,y}[\rho\left( t, \cdot,y\right)]_{\beta}\left\vert u \right\vert_{1+\beta',\infty}+\sup_{t,z,y}\left\vert \rho\left( t, z,y\right)\right\vert\left\vert u\right\vert_{1+\beta,\infty}\right)$.
\end{proof}

\begin{lemma}\label{aa}
	Let $\nu$ be a L\'{e}vy measure satisfying \textbf{$\mathbf{\tilde{A}}(w,l,\gamma)$} and $\rho$ satisfy \textbf{H($K,\beta$)}. $\beta\in\left(0,1/\alpha\right)$. Then for any $u\in \tilde{C}_{\infty ,\infty }^{1+\beta }\left(\mathbf{R}^d\right)$ and any $\varepsilon\in\left( 0,1\right)$,
	\begin{eqnarray}\label{cross}
	\sup_{t,z}\left\vert \langle u,\eta_{m,z}\rangle_{t,z}\right\vert_{\beta,\infty}\leq C l\left(m\right)^{1+\beta}\left( \varepsilon \left\vert u\right\vert_{1+\beta,\infty}+C_{\varepsilon}\left\vert u\right\vert_{0}\right),
	\end{eqnarray}
	where $C_{\varepsilon}$ depends on $\varepsilon$ but is independent of $u$.
\end{lemma}
\begin{proof}
	Direct computation shows that for $\kappa\in\left( 0,1\right)$,
	\begin{eqnarray*}
		L^{\mu,\kappa}\eta_{m,z}\left(x\right)=w\left(m^{-1}\right)^{-\kappa}L^{\tilde{\mu}_{m^{-1}},\kappa}\eta\left(m\left( x-z\right)\right).
	\end{eqnarray*}
	
	By \textbf{$\mathbf{\tilde{A}}(w,l,\gamma)$}, there is $\kappa\in\left(1/2,1\right)$ such that $\int_1^{\infty}t^{\kappa-1}\gamma\left( t\right)^{-1}dt<\infty$. Apply $\eqref{mod1}$ with such a $\kappa$.
	\begin{eqnarray*}
		&&\left\vert\langle u,\eta_{m,z}\rangle_{t,z}\right\vert_0\\ &\leq& C\int\left\vert u\left(  x+y\right)-u\left(  x\right)\right\vert\left\vert \eta_{m,z}\left(  x+y\right)-\eta_{m,z}\left(  x\right)\right\vert\nu\left( dy\right)\\
		&\leq& Cw\left(m^{-1}\right)^{-\kappa}\left\vert L^{\mu,\kappa}u\right\vert_0\left\vert L^{\tilde{\mu}_{m^{-1}},\kappa}\eta\right\vert_0\int_{\left\vert y\right\vert\leq 1} w\left(\left\vert  y\right\vert\right)^{2\kappa}\nu\left( dy\right)+C\left\vert u\right\vert_0.
	\end{eqnarray*}
	According to Lemmas \ref{ess}, \ref{cont} and \cite[Proposition 4]{mf},
	\begin{eqnarray*}
		\left\vert\langle u,\eta_{m,z}\rangle_{t,z}\right\vert_0&\leq& Cw\left(m^{-1}\right)^{-\kappa}\left(\left\vert L^{\mu,\kappa}u\right\vert_0+\left\vert u\right\vert_0\right)\\
		&\leq& C l\left(m\right)^{\kappa}\left( \varepsilon \left\vert u\right\vert_{1+\beta,\infty}+C_{\varepsilon}\left\vert u\right\vert_{0}\right).
	\end{eqnarray*}
	
	For the difference estimate, let us set $a=\left\vert x_1-x_2\right\vert$ and denote
	\begin{eqnarray*}
		&& \left\vert \langle u,\eta_{m,z}\rangle_{t,z}\left( x_1\right)-\langle u,\eta_{m,z}\rangle_{t,z}\left( x_2\right)\right\vert\\
		&\leq& \vert\int_{\left\vert y\right\vert\leq 1}\left[ u\left(  x_1+y\right)-u\left(  x_1\right)- u\left(  x_2+y\right)+u\left(  x_2\right)\right]\\
		&&\quad\quad\cdot\left[ \eta_{m,z}\left(  x_1+y\right)-\eta_{m,z}\left(  x_1\right)\right]\rho\left( t,z,y\right) \nu\left( dy\right)\vert\\
		&+& \vert\int_{\left\vert y\right\vert\leq 1}\left[ u\left(  x_2+y\right)-u\left(  x_2\right)\right]\lbrack \eta_{m,z}\left(  x_1+y\right)-\eta_{m,z}\left(  x_1\right)\\
		&&\quad\quad-\eta_{m,z}\left(  x_2+y\right)+\eta_{m,z}\left(  x_2\right)\rbrack\rho\left(t, z,y\right) \nu\left( dy\right)\vert\\
		&+& \vert\int_{\left\vert y\right\vert> 1}\{\left[ u\left(  x_1+y\right)-u\left(  x_1\right)\right]\left[ \eta_{m,z}\left(  x_1+y\right)-\eta_{m,z}\left(  x_1\right)\right]\\
		&&\quad\quad-\left[ u\left(  x_2+y\right)-u\left(  x_2\right)\right]\left[ \eta_{m,z}\left(  x_2+y\right)-\eta_{m,z}\left(  x_2\right)\right]\}\rho\left( t,z,y\right) \nu\left( dy\right)\vert\\
		&:=& I_1+I_2+I_3.
	\end{eqnarray*}
	Similarly, we use $\eqref{mod1}$. Then for some $\kappa\in\left(1/2,1\right)$,
	\begin{eqnarray*}
		I_1 &\leq&Cw\left(m^{-1}\right)^{-\kappa}\sup_z\left\vert L^{\mu,\kappa}u\left( x_1+z\right)-L^{\mu,\kappa}u\left( x_2+z\right)\right\vert\left\vert L^{\tilde{\mu}_{m^{-1}},\kappa}\eta\right\vert_0\\
		&\leq& Cw\left(m^{-1}\right)^{-\kappa}w\left( a\right)^{\beta}\left\vert L^{\mu,\kappa}u\right\vert_{\beta,\infty}\\
		&\leq& C l\left(m\right)^{\kappa}w\left( a\right)^{\beta}\left( \varepsilon \left\vert u\right\vert_{1+\beta,\infty}+C_{\varepsilon}\left\vert u\right\vert_{0}\right).
	\end{eqnarray*}
	For the same $\kappa$,
	\begin{eqnarray*}
		I_2 &\leq&Cw\left(m^{-1}\right)^{-\kappa}\sup_z\left\vert L^{\tilde{\mu}_{m^{-1}},\kappa}\eta\left( m\left(x_1-z\right)\right)-L^{\tilde{\mu}_{m^{-1}},\kappa}\eta\left( m\left(x_2-z\right)\right)\right\vert\left\vert L^{\mu,\kappa}u\right\vert_0\\
		&\leq& C w\left(m^{-1}\right)^{-\kappa}l\left(m\right)^{\beta}w\left( a\right)^{\beta}\left\vert u\right\vert_{\kappa+\beta,\infty}\\
		&\leq& C l\left(m\right)^{1+\beta}w\left( a\right)^{\beta}\left( \varepsilon \left\vert u\right\vert_{1+\beta,\infty}+C_{\varepsilon}\left\vert u\right\vert_{0}\right).
	\end{eqnarray*}
	
	Besides,
	\begin{eqnarray*}
		I_3 &\leq& C\int_{\left\vert y\right\vert> 1}\vert\left[u\left(  x_1+y\right)-u\left(  x_1\right)\right]\lbrack\eta_{m,z}\left(  x_1+y\right)-\eta_{m,z}\left(  x_1\right)\\
		&&\qquad\qquad-\eta_{m,z}\left(  x_2+y\right)+\eta_{m,z}\left(  x_2\right)\rbrack\vert \nu\left( dy\right)\\
		&&+C\int_{\left\vert y\right\vert> 1}\vert\left[u\left(  x_1+y\right)- u\left(  x_2+y\right)-u\left(  x_1\right)+u\left(  x_2\right)\right]\\
		&&\qquad\qquad\left[ \eta_{m,z}\left(  x_2+y\right)-\eta_{m,z}\left(  x_2\right)\right]\vert\nu\left( dy\right)\\
		&:=&I_{31}+I_{32},
	\end{eqnarray*}
	where
	\begin{eqnarray*}
		I_{31} &\leq& C\left\vert u\right\vert_0\int_{\left\vert y\right\vert> 1}\vert\eta_{m,z}\left(  x_1+y\right)-\eta_{m,z}\left(  x_2+y\right)-\eta_{m,z}\left(  x_1\right)\\
		&&+\eta_{m,z}\left( x_2\right)\vert \nu\left( dy\right)
		\leq  C\left\vert u\right\vert_0l\left(m\right)^{\beta}w\left( a\right)^{\beta},
	\end{eqnarray*}
	and obviously,
	\begin{eqnarray*}
		I_{32} &\leq&  Cw\left( a\right)^{\beta}\left\vert u\right\vert_{\beta,\infty}\leq  Cw\left( a\right)^{\beta}\left( \varepsilon \left\vert u\right\vert_{1+\beta,\infty}+C_{\varepsilon}\left\vert u\right\vert_{0}\right).
	\end{eqnarray*}
	
	Summarizing, for any $z\in\mathbf{R}^d$, 
	\begin{eqnarray*}
		\left[\langle u,\eta_{m,z}\rangle_{t,z}\right]_{\beta}\leq Cl\left(m\right)^{1+\beta}\left( \varepsilon \left\vert u\right\vert_{1+\beta,\infty}+C_{\varepsilon}\left\vert u\right\vert_{0}\right).
	\end{eqnarray*}
	
	It follows immediately from \cite[Proposition 1]{mf} that
	\begin{eqnarray*}
		&&\sup_{t,z}\left\vert \langle u,\eta_{m,z}\rangle_{t,z}\right\vert_{\beta,\infty}\\
		&\leq& C\sup_{t,z}\left\vert \langle u,\eta_{m,z}\rangle_{t,z}\right\vert_{\beta}\leq C l\left(m\right)^{1+\beta}\left( \varepsilon \left\vert u\right\vert_{1+\beta,\infty}+C_{\varepsilon}\left\vert u\right\vert_{0}\right).
	\end{eqnarray*}
\end{proof}

\subsection{Operators with Space-Dependent Coefficients}
In this section, we study the operator
\begin{eqnarray*}
\mathcal{G}\varphi\left( x\right)&:=&\int\left[\varphi\left( x+G\left( x\right)y\right)-\varphi\left( x\right)-\chi_{\alpha}\left( y\right)G\left( x\right)y\cdot \nabla \varphi\left( x\right) \right]\nu\left(d y\right).
\end{eqnarray*}

We define the norm of an $d\times d$-invertible matrix function $G\left( x\right), x\in\mathbf{R}^d$ to be its operator norm, i.e., 
\begin{equation*}
\left\vert G\left( x\right) \right\vert :=\sup_{y\in \mathbf{R}%
	^{d},\left\vert y\right\vert =1}\left\vert G\left( x\right) y\right\vert ,
\end{equation*}%
and 
\begin{equation*}
\Vert G\Vert :=\sup_{x\in \mathbf{R}^{d}}\left\vert G\left( x\right)
\right\vert .
\end{equation*}%
 If all entries of $G$ are constants, then $G$ is viewed as a constant function and definitions above apply. Note $\left\Vert G\right\Vert $ being finite implies finiteness of each
 entry. If furthermore $\left\vert \det G\left( z\right)\right\vert\geq c_0$ for some $c_0>0$, then $\left\Vert G^{-1}\right\Vert$ is also finite.

\begin{lemma}\label{lem1}
	Let $G$ be an invertible $d\times d$-matrix and $\beta>0$. $f\in\tilde{C}^{\beta}_{\infty,\infty}\left(\mathbf{R}^d\right)$. $g\left(x\right):=f\left( Gx\right), x\in\mathbf{R}^d$. Then,
	\begin{eqnarray}
	\left\vert g\right\vert_{\beta,\infty}\leq C\left\vert f\right\vert_{\beta,\infty}
	\end{eqnarray}
	for some $C$ only depending on $\left\Vert G^{-1}\right\Vert$ and $\left\Vert G\right\Vert$.
\end{lemma}
\begin{proof}
	Consider the mapping $T:\mathbf{R}^d\to \mathbf{R}^d$ such that $T\left( x\right)=Gx$. Then $T^{-1}\left( x\right)=G^{-1}x$. Clearly, both $T$ and $T^{-1}$ are continuous and $ \left\Vert T\right\Vert^{-1}=\left\Vert G\right\Vert^{-1}\leq \left\Vert T^{-1}\right\Vert=\left\Vert G^{-1}\right\Vert$. For any $j\neq 0$, 
	\begin{eqnarray*}
		\left\vert g\ast \varphi_j\right\vert _{0}&=& \sup_x\left\vert\int f\left( Gy\right)\varphi_{j} \left( x-y\right)dy\right\vert\\
		&=& \sup_x\frac{1}{\left\vert \det G\right\vert}\left\vert\int f\left( y\right)\varphi_{j} \left( G^{-1}x-G^{-1}y\right)dy\right\vert\\
		&=& \left\vert\mathcal{F}^{-1}\left[\phi_j\left( G\xi\right)\mathcal{F}f\right]\right\vert_0.
	\end{eqnarray*}
	Note $\phi_j\left( G\cdot\right)$ is supported on $\{\xi: N^{j-1}\leq \left\vert G\xi\right\vert\leq N^{j+1}\}\subset \{\xi: N^{j-1}\left\Vert G\right\Vert^{-1}\leq \left\vert \xi\right\vert\leq N^{j+1}\left\Vert G^{-1}\right\Vert \}$. Denote 
	\begin{eqnarray*}
		n\left(j\right)&=&\min \{i:N^{j+1}\left\Vert G^{-1}\right\Vert\leq N^{i}\}\vee 1 ,\\
		m\left(j\right)&=&\max \{i:N^{i}\leq N^{j-1}\left\Vert G\right\Vert^{-1}\}\vee 0.
	\end{eqnarray*}
	Then $n\left(j\right)=n\left(1\right)+j-1$, $m\left(j\right)=m\left(1\right)+j-1$, and that $n\left(j\right)-m\left(j\right)\leq n\left(1\right)-m\left(1\right)+1$ which is independent of $j$. Moreover, 
	\begin{eqnarray*}
		\phi_j\left( G\xi\right)=\phi_j\left( G\xi\right)\sum_{i=m\left(j\right)}^{n\left(j\right)}\phi_i\left(\xi\right).
	\end{eqnarray*}
	Therefore,
	\begin{eqnarray*}
		\left\vert g\ast \varphi_j\right\vert _{0}
		&=& \left\vert\mathcal{F}^{-1}\left[\phi_j\left( G\xi\right)\sum_{i=m\left(j\right)}^{n\left(j\right)}\phi_i\left(\xi\right)\mathcal{F}f\right]\right\vert_0\\
		&\leq& \left(n\left(1\right)-m\left(1\right)+1\right)\sup_j\left\vert f\ast \varphi_j\right\vert _{0} \left\vert \mathcal{F}^{-1}\left[\phi_j\left( G\xi\right)\right]\right\vert _{L^1\left(\mathbf{R}^d\right)}\\
		&\leq& Cw\left( N^{-j}\right)^{\beta}\left\vert f\right\vert_{\beta,\infty}.
	\end{eqnarray*}
	
	Similarly, if $j=0$,
	\begin{eqnarray*}
		\left\vert g\ast \varphi_0\right\vert _{0}&=& \left\vert\mathcal{F}^{-1}\left[\widehat{\varphi_0}\left( G\xi\right)\mathcal{F}f\right]\right\vert_0,
	\end{eqnarray*}
	and $supp \left(\widehat{\varphi_0}\left( G\xi\right)\right)=\{\xi: \left\vert G\xi\right\vert\leq N\}\subset \{\xi: \left\vert \xi\right\vert\leq N\left\Vert G^{-1}\right\Vert \}$. Denote $k=\min \{i:N\left\Vert G^{-1}\right\Vert\leq N^{i}\}\vee 1$. Then, 
	\begin{eqnarray*}
		\widehat{\varphi_0}\left( G\xi\right)=\widehat{\varphi_0}\left( G\xi\right)\left(\widehat{\varphi_0}\left( \xi\right)+\sum_{i=1}^{k}\phi_i\left(\xi\right)\right).
	\end{eqnarray*}
	Therefore,
	\begin{eqnarray*}
		\left\vert g\ast \varphi_j\right\vert _{0}
		&=& \left\vert\mathcal{F}^{-1}\left[\widehat{\varphi_0}\left( G\xi\right)\left(\widehat{\varphi_0}\left( \xi\right)+\sum_{i=1}^{k}\phi_i\left(\xi\right)\right)\mathcal{F}f\right]\right\vert_0\\
		&\leq& C\sup_j\left\vert f\ast \varphi_j\right\vert _{0} \left\vert \mathcal{F}^{-1}\left[\widehat{\varphi_0}\left( G\xi\right)\right]\right\vert _{L^1\left(\mathbf{R}^d\right)}\\
		&\leq& Cw\left( N^{-j}\right)^{\beta}\left\vert f\right\vert_{\beta,\infty}.
	\end{eqnarray*}
	
	Summarizing, $\left\vert g\right\vert_{\beta,\infty}\leq C\left\vert f\right\vert_{\beta,\infty}$.
\end{proof}

\begin{proposition}\label{ppr1}
	Let $\nu$ be a L\'{e}vy measure satisfying \textbf{A(w,l)} and $G$ be an invertible $d\times d$-matrix. For any function $f\in  C_b^{2}\left(\mathbf{R}^d\right)$, 
	\begin{eqnarray*}
		Lf\left( x\right)&:=&\int\left[f\left( x+Gy\right)-f\left( x\right)-\chi_{\alpha}\left( y\right)Gy\cdot \nabla f\left( x\right) \right]\nu\left(d y\right).
	\end{eqnarray*}
	Then for $\beta\in\left( 0,1/\alpha\right)$, there exists $C$ depending on $\left\Vert G^{-1}\right\Vert$ and $\left\Vert G\right\Vert$ such that
	\begin{eqnarray*}
		\left\vert L f\right\vert_{\beta,\infty}\leq C\left\vert  f\right\vert_{1+\beta,\infty}.
	\end{eqnarray*}
\end{proposition}
\begin{proof}
	If $g\left(x\right):=f\left( Gx\right)$. Then $Lf\left( Gx\right)=L^{\nu}g\left( x\right)$. By previous continuity and equivalence results and Lemma \ref{lem1}, 
	\begin{eqnarray*}
		\left\vert Lf\right\vert_{\beta,\infty}\leq C\left\vert Lf\left( G\cdot\right)\right\vert_{\beta,\infty}=\left\vert L^{\nu}g \right\vert_{\beta,\infty} \leq C\left\vert g \right\vert_{1+\beta,\infty}\leq C\left\vert f \right\vert_{1+\beta,\infty}.
	\end{eqnarray*}	
\end{proof}

Let us denote 
\begin{eqnarray}
	\nabla^{\alpha,z}u\left( x;y\right)&=&u\left( x+G\left( z\right)y\right)-u\left( x\right)-\chi_{\alpha}\left( y\right)G\left( z\right)y\cdot \nabla u\left( x\right),\notag\\
	L_{z} u\left(x\right)&=&\int \nabla^{\alpha,z}u\left( x;y\right)\nu\left( dy\right),\label{opp1}\\
	\bar{G}_{z,z'}&:=&\left\Vert G\left( z\right)-G\left( z'\right)\right\Vert,\notag
\end{eqnarray}
and
\begin{eqnarray*}
	g\left(z,z'\right)
	= \left\{ 
	\begin{array}{cc}
		w\left(\bar{G}_{z,z'}^{-1}\right)^{-1} & \text{if } \alpha\in\left( 0,1\right) , \\ 
		w\left( \bar{G}_{z,z'}^{-1}\right)^{-1}w\left( \bar{G}_{z,z'}\right)^{-\delta'} \vee \bar{G}_{z,z'} & \text{if } \alpha=1, \\ 
		\bar{G}_{z,z'} & \text{if } \alpha\in\left( 1,2\right) . 
	\end{array}%
	\right. 
\end{eqnarray*}
\begin{lemma}\label{prr1}
	Let $\beta\in\left( 0,1/\alpha\right)$, $\nu$ be a L\'{e}vy measure satisfying \textbf{$\mathbf{\tilde{A}}(w,l,\gamma)$}, and $G\left( z\right),\forall z\in\mathbf{R}^d$ satisfy \textbf{G($c_0, K,\beta$)}. If $u\in\tilde{C}^{1+\beta}\left( \mathbf{R}^d\right)$, then,
	\begin{eqnarray*}
		\left\vert L_{z} u-L_{z'} u\right\vert_{0}&\leq& Cg\left(z,z'\right)\left\vert  u\right\vert_{1+\beta',\infty},\\
		\left[ L_{z} u-L_{z'} u\right]_{\beta}&\leq& C \left(\bar{G}_{z,z'}\right)^{\sigma}\left\vert u\right\vert_{1+\beta,\infty}
	\end{eqnarray*}
	for some $\beta'\in\left( 0,\beta\right),\sigma\in\left( 0,1\right)$. $C$ is independent $z,z'$ and $u$.
\end{lemma}
\begin{proof}
	Write for simplicity $G=G\left( z\right)$, $G'=G\left( z'\right)$ and $\bar{G}=\bar{G}_{z,z'}$. Use $\eqref{tran2}$.
	\begin{eqnarray*}
	&& \nabla^{\alpha,z}u\left( x;y\right)-\nabla^{\alpha,z'}u\left( x;y\right)\notag\\
	&=& C w\left( R\right)^{\kappa}\int_0^{\infty}t^{\kappa-1} \int L^{\mu,\kappa}u\left( x+R\vartheta\right)\notag\\
	&&\cdot\left[ \nabla^{\alpha,z}p^{R}\left( t,\vartheta;-R^{-1}Gy\right)-\nabla^{\alpha,z'}p^{R}\left( t,\vartheta;-R^{-1}G'y\right)\right] d\vartheta dt\notag\\
	&:=& D\left( \kappa, R\right)
	\end{eqnarray*}
	with $\kappa$ and $R$ to be determined according to our needs. If $\alpha\in\left( 0,1\right)$, by \textbf{$\mathbf{\tilde{A}}(w,l,\gamma)$}, we may split the integral as follows.
	\begin{eqnarray*}
		&&\left\vert L_{z} u\left( x\right)-L_{z'} u\left( x\right)\right\vert \\
		&\leq& \int_{\bar{G}\left\vert y\right\vert\leq 1}   \left\vert D\left( \kappa, \bar{G}\left\vert y\right\vert\right) \right\vert\nu\left(dy\right)+\int_{\bar{G}\left\vert y\right\vert>1}   \left\vert u\left( x+Gy\right)-u\left( x+G'y\right) \right\vert\nu\left(dy\right)\\
		&:=& I_1+I_2
	\end{eqnarray*}
for some $\kappa\in\left( 1,1+\delta\right)$. By Lemmas \ref{ess} and \ref{lemma3} in Appendix,
\begin{eqnarray*}
	I_1&\leq& C\left\vert L^{\mu,\kappa}u\right\vert_0\int_{\bar{G}\left\vert y\right\vert\leq 1}\int_{0}^{\infty }t^{\kappa-1}\left( 1\wedge \gamma \left( t\right)^{-1}\bar{G}^{-1}\left\vert y\right\vert^{-1}\left\vert Gy-G'y\right\vert\right)dt \\
	 &&\quad w\left( \bar{G}\left\vert y\right\vert\right)^{\kappa}\nu\left(dy\right)\\
	&\leq& Cw\left( \bar{G}^{-1}\right)^{-1}\left\vert L^{\mu,\kappa}u\right\vert_0\int_{\left\vert y\right\vert\leq 1}   w\left( \left\vert y\right\vert\right)^{\kappa}\tilde{\nu}_{\bar{G}^{-1}}\left(dy\right)\\
	 &\leq& Cw\left( \bar{G}^{-1}\right)^{-1}\left\vert L^{\mu,\kappa}u\right\vert_0. 
 \end{eqnarray*}
Besides,
\begin{eqnarray*}
	I_2 &\leq&  Cw\left( \bar{G}^{-1}\right)^{-1}\int_{\left\vert y\right\vert> 1}\left\vert u\right\vert_0\tilde{\nu}_{\bar{G}^{-1}}\left(dy\right) \leq Cw\left( \bar{G}^{-1}\right)^{-1}\left\vert  u\right\vert_{0}. 
\end{eqnarray*}
Therefore when $\alpha\in\left( 0,1\right)$, there exists $\beta'\in\left( 0,\beta\right)$ such that
\begin{eqnarray}\label{s1}
\left\vert L_{z} u-L_{z'} u\right\vert_0 
\leq Cw\left( \bar{G}^{-1}\right)^{-1}\left\vert  u\right\vert_{1+\beta',\infty}.
\end{eqnarray}

If $\alpha=1$, we write instead
\begin{eqnarray*}
&&\left\vert L_{z} u\left( x\right)-L_{z'} u\left( x\right)\right\vert\\
&\leq& \int_{\left\vert y\right\vert\leq 1}   \left\vert D\left( 1+\delta, \left\vert y\right\vert\right) \right\vert\nu\left(dy\right)+\int_{\bar{G}\left\vert y\right\vert\leq 1,\left\vert y\right\vert>1}   \left\vert u\left( x+Gy\right)-u\left( x+G'y\right) \right\vert\nu\left(dy\right)\\
&+&\int_{\bar{G}\left\vert y\right\vert> 1,\left\vert y\right\vert>1}   \left\vert u\left( x+Gy\right)-u\left( x+G'y\right) \right\vert\nu\left(dy\right):=I_3+I_4+I_5,
\end{eqnarray*}
where
\begin{eqnarray*}
	I_3&\leq& C\left\vert L^{\mu,1+\delta}u\right\vert_0\int_{\left\vert y\right\vert\leq 1}w\left( \left\vert y\right\vert\right)^{1+\delta}\int_{0}^{\infty }t^{\kappa-1}\bar{G}\left( \gamma \left( t\right)^{-1}\wedge \gamma \left( t\right)^{-2}\right)dt\nu\left(dy\right) \\
	&\leq& C\bar{G}\left\vert L^{\mu,1+\delta}u\right\vert_0.
\end{eqnarray*}
Meanwhile, similarly as $I_1,I_2$, we have
\begin{eqnarray*}
	&&w\left( \bar{G}^{-1}\right)\left(I_4+I_5\right)\\
	&\leq&   C\left(\left\vert L^{\mu,1-\delta'}u\right\vert_0\int_{\left\vert y\right\vert\leq 1,\left\vert y\right\vert>\bar{G}}   w\left( \left\vert y\right\vert\right)^{1-\delta'}\tilde{\nu}_{\bar{G}^{-1}}\left(dy\right) +\int_{\left\vert y\right\vert> 1}\left\vert u\right\vert_0\tilde{\nu}_{\bar{G}^{-1}}\left(dy\right) \right)\\
	&\leq&   C\left(\frac{1}{\delta'}\left\vert L^{\mu,1-\delta'}u\right\vert_0+\frac{1}{\delta'}\left\vert L^{\mu,1-\delta'}u\right\vert_0w\left( \bar{G}\right)^{-\delta'}+\left\vert u\right\vert_0 \right).
\end{eqnarray*}
Thus for $\alpha=1$, there is $\beta'\in\left( 0,\beta\right)$ so that
\begin{eqnarray}\label{s2}
\left\vert L_{z} u-L_{z'} u\right\vert_0 
\leq \frac{C}{\delta'}\left(\bar{G}\vee w\left( \bar{G}^{-1}\right)^{-1}w\left( \bar{G}\right)^{-\delta'}\right)\left\vert  u\right\vert_{1+\beta',\infty}.
\end{eqnarray}

Next, we discuss the case $\alpha\in\left( 1,2\right)$. Split the integral as
\begin{eqnarray*}
	&&\left\vert L_{z} u\left( x\right)-L_{z'} u\left( x\right)\right\vert\\
	&\leq& \int_{\left\vert y\right\vert\leq 1}   \left\vert D\left( \kappa, \left\vert y\right\vert\right) \right\vert\nu\left(dy\right)+\int_{\left\vert y\right\vert> 1}   \left\vert \nabla^{\alpha,z}u\left( x;y\right)-\nabla^{\alpha,z'}u\left( x;y\right) \right\vert\nu\left(dy\right)\\
	&:=& I_6+I_7.
\end{eqnarray*}
Then as how we estimated $I_3$, we have $I_6\leq C\bar{G}\left\vert L^{\mu,\kappa}u\right\vert_0$ for some $\kappa\in\left( 1,1+\delta\right)$.
Clearly, $I_7\leq C\bar{G}\left\vert \nabla u\right\vert_0$. Thus for $\alpha\in\left( 1,2\right)$, there is $\beta'\in\left( 0,\beta\right)$ so that
\begin{eqnarray}\label{s3}
\left\vert L_{z} u-L_{z'} u\right\vert_0 
\leq C\bar{G}\left\vert  u\right\vert_{1+\beta',\infty}.
\end{eqnarray}

We now estimate the difference. Without loss of generality, we set $\left\vert x_1-x_2\right\vert=a\in\left( 0,1\right)$. By Lemma \ref{diff},
\begin{eqnarray*}
&& \nabla^{\alpha,z}u\left( x_1;y\right)-\nabla^{\alpha,z'}u\left( x_1;y\right)-\nabla^{\alpha,z}u\left( x_2;y\right)+\nabla^{\alpha,z'}u\left( x_2;y\right)\notag\\
&=& C w\left( R\right)^{\kappa}\int_0^{\infty}t^{\kappa-1} \int\left[ L^{\mu,\kappa}u\left( x_1+R\vartheta\right)-L^{\mu,\kappa}u\left( x_2+R\vartheta\right)\right]\notag\\
&&\cdot\left[ \nabla^{\alpha,z}p^{R}\left( t,\vartheta;-R^{-1}y\right)-\nabla^{\alpha,z'}p^{R}\left( t,\vartheta;-R^{-1}y\right)\right] d\vartheta dt\notag\\
&:=& \widetilde{D}\left( \kappa, R\right)
\end{eqnarray*}
with $\kappa$ and $R$ to be determined. Then, 
\begin{eqnarray*}
	&&\left\vert L_{z} u\left(x_1\right)-L_{z'} u\left(x_1\right)-L_{z} u\left(x_2\right)+L_{z'} u\left(x_2\right)\right\vert\\
	&\leq& \int_{\left\vert y\right\vert\leq a}   \left\vert \widetilde{D}\left( \kappa,\left\vert y\right\vert\right)\right\vert \nu\left(dy\right)+ \int_{\left\vert y\right\vert> a}  \left\vert \widetilde{D}\left( \kappa',\left\vert y\right\vert\right)\right\vert \nu\left(dy\right)\\
	&:=& I_8+I_9.
\end{eqnarray*}
Denote $\varsigma\left( r\right)=\nu\left( \left\vert y\right\vert >r\right)$. By Lemma \ref{lemma3} and \cite[Lemma 1]{mf}, for all $\alpha\in\left( 0,2\right)$, there is $\beta'\in\left(0,\beta\right)$ and $\sigma\in\left(0,1\right)$ such that
\begin{eqnarray*}
	I_8 &\leq& -C\left[L^{\mu,1+\beta'}u\right]_{\beta-\beta'}\bar{G}^{\sigma}w\left( a\right)^{\beta-\beta'}\int_{0}^a \varsigma\left( r\right)^{-1-\beta'} d\varsigma\left( r\right)\\
	&\leq& C\left\vert u\right\vert_{1+\beta,\infty}\bar{G}^{\sigma}w\left( a\right)^{\beta-\beta'} \varsigma\left( r\right)^{-\beta'}|_{0}^{a}\leq C\left\vert u\right\vert_{1+\beta,\infty}\bar{G}^{\sigma} w\left( a\right)^{\beta}.
\end{eqnarray*}
Recall \textbf{$\mathbf{\tilde{A}}(w,l,\gamma)$}. Using the symmetry assumption for $\alpha=1$ and non-degeneracy of $G$, we can set $\kappa=\beta+\min\left( \delta,\varepsilon\right)/2$ if $\alpha\neq 1$ and $\kappa=\beta+\left( \delta'+\varepsilon\right)/2$ if $\alpha=1$, there is $\beta'\in\left(0,\beta\right)$ and $\sigma\in\left(0,1\right)$ such that
\begin{eqnarray*}
	I_9	&\leq& C\left[ L^{\mu,1+\beta-\kappa}u\right]_{\kappa}\bar{G}^{\sigma}w\left( a\right)^{\kappa}\int_{\left\vert y\right\vert> a}w\left( \left\vert y\right\vert\right)^{1+\beta-\kappa}\nu\left( dy\right)\\
	&\leq&- C\left\vert u\right\vert_{1+\beta,\infty}\bar{G}^{\sigma}w\left( a\right)^{\kappa}\int_{ a}^{\infty}\varsigma\left( r\right)^{-1-\beta+\kappa}d\varsigma\left(r\right)\\
	&\leq& C \left\vert u\right\vert_{1+\beta,\infty}\bar{G}^{\sigma}w\left( a\right)^{\kappa}\varsigma\left(a\right)^{\kappa-\beta}\leq  C \left\vert u\right\vert_{1+\beta,\infty}\bar{G}^{\sigma}w\left( a\right)^{\beta}.
\end{eqnarray*}

This ends the proof.
\end{proof}

\begin{corollary}\label{col}
	Let $\beta\in\left( 0,1/\alpha\right)$, $\nu$ be a L\'{e}vy measure satisfying \textbf{$\mathbf{\tilde{A}}(w,l,\gamma)$}, and $G\left( z\right),\forall z\in\mathbf{R}^d$ satisfy \textbf{G($c_0, K,\beta$)}. If $u\in\tilde{C}^{1+\beta}\left( \mathbf{R}^d\right)$, then,
	\begin{eqnarray*}
		\left\vert \mathcal{G} u-L_{z} u\right\vert_{\beta} &\leq& C\left(\left(\bar{G}_{x,z}\right)^{\sigma}\left\vert u\right\vert_{1+\beta,\infty}+ \left\vert u\right\vert_{1+\beta',\infty} \right)
	\end{eqnarray*}
	for some $\beta'\in\left( 0,\beta\right)$. $C$ is independent of $x,z$ and $u$.
\end{corollary}
\begin{proof}
First by Lemma \ref{prr1}, $\left\vert \mathcal{G} u-L_{z} u\right\vert_{0}\leq \sup_{z,z'}\left\vert L_{z'} u-L_{z} u\right\vert_{0}\leq C\left\vert  u\right\vert_{1+\beta',\infty}$ for some $\beta'\in\left( 0,\beta\right)$. In the meantime,
\begin{eqnarray*}
	&&\left\vert L_{x+y} u\left( x+y\right)-L_{z} u\left( x+y\right)-L_{x} u\left( x\right)+L_{z} u\left( x\right)\right\vert\\
	&\leq&\left\vert L_{x} u\left( x+y\right)-L_{z} u\left( x+y\right)-L_{x} u\left( x\right)+L_{z} u\left( x\right)\right\vert\\
	&& +\left\vert L_{x+y} u\left( x+y\right)-L_{x} u\left( x+y\right)\right\vert\\
	&\leq& C \left(\bar{G}_{x,z}\right)^{\sigma}\left\vert u\right\vert_{1+\beta,\infty} w\left( \left\vert y\right\vert\right)^{\beta}+ C \left\vert u\right\vert_{1+\beta',\infty} w\left( \left\vert y\right\vert\right)^{\beta}.
\end{eqnarray*} 
Namely, $\left[\mathcal{G} u-L_{z} u\right]_{\beta}\leq C\left(\left(\bar{G}_{x,z}\right)^{\sigma}\left\vert u\right\vert_{1+\beta,\infty}+ \left\vert u\right\vert_{1+\beta',\infty} \right)$. 
\end{proof}

For $\eta_{m,z}$ introduced in previous section, we denote
\begin{eqnarray}\label{opp2}
	&&\langle u,\eta_{m,z}\rangle_{z}\\
	&=&\int\left[ u\left(  x+G\left( z\right)y\right)-u\left(  x\right)\right]\left[ \eta_{m,z}\left(  x+G\left( z\right)y\right)-\eta_{m,z}\left(  x\right)\right]\nu\left( dy\right).\notag
\end{eqnarray}

\begin{lemma}\label{aa2}
	Let $\nu$ be a L\'{e}vy measure satisfying \textbf{$\mathbf{\tilde{A}}(w,l,\gamma)$} and $\left\Vert G\left( z\right)\right\Vert\leq K, \forall z\in\mathbf{R}^d$ for some $K>0$. $\beta\in\left(0,1/\alpha\right)$. Then for any $u\in \tilde{C}_{\infty ,\infty }^{1+\beta }\left(\mathbf{R}^d\right)$ and any $\varepsilon\in\left( 0,1\right)$,
	\begin{eqnarray}\label{cros}
	\sup_{z}\left\vert \langle u,\eta_{m,z}\rangle_{z}\right\vert_{\beta,\infty}\leq C l\left(m\right)^{1+\beta}\left( \varepsilon \left\vert u\right\vert_{1+\beta,\infty}+C_{\varepsilon}\left\vert u\right\vert_{0}\right),
	\end{eqnarray}
	where $C_{\varepsilon}$ depends on $\varepsilon$ but is independent of $u$.
\end{lemma}
\begin{proof}
We proceed in the same manner as in Lemma \ref{aa}. First, since $\left\Vert G\left( z\right)\right\Vert$ is uniformly bounded, there is $\kappa\in\left(1/2,1\right)$ such that
\begin{eqnarray*}
	&&\left\vert\langle u,\eta_{m,z}\rangle_{z}\right\vert_0\\ &\leq& C\int\left\vert u\left(  x+ G\left( z\right)y\right)-u\left(  x\right)\right\vert\left\vert \eta_{m,z}\left(  x+G\left( z\right)y\right)-\eta_{m,z}\left(  x\right)\right\vert\nu\left( dy\right)\\
	&\leq& Cw\left(m^{-1}\right)^{-\kappa}\left\vert L^{\mu,\kappa}u\right\vert_0\left\vert L^{\tilde{\mu}_{m^{-1}},\kappa}\eta\right\vert_0\int_{\left\vert y\right\vert\leq 1} w\left(\left\vert  y\right\vert\right)^{2\kappa}\nu\left( dy\right)+C\left\vert u\right\vert_0\\
	&\leq& C l\left(m\right)^{\kappa}\left( \varepsilon \left\vert u\right\vert_{1+\beta,\infty}+C_{\varepsilon}\left\vert u\right\vert_{0}\right).
\end{eqnarray*}

For the difference, again, let us set $a=\left\vert x_1-x_2\right\vert\in \left( 0,1\right)$ and estimate
\begin{eqnarray*}
	&& \left\vert \langle u,\eta_{m,z}\rangle_{z}\left( x_1\right)-\langle u,\eta_{m,z}\rangle_{z}\left( x_2\right)\right\vert\\
	&\leq& \vert\int_{\left\vert y\right\vert\leq 1}\left[ u\left(  x_1+G\left( z\right)y\right)-u\left(  x_1\right)- u\left(  x_2+G\left( z\right)y\right)+u\left(  x_2\right)\right]\\
	&&\quad\quad\cdot\left[ \eta_{m,z}\left(  x_1+G\left( z\right)y\right)-\eta_{m,z}\left(  x_1\right)\right]\nu\left( dy\right)\vert\\
	&+& \vert\int_{\left\vert y\right\vert\leq 1}\left[ u\left(  x_2+G\left( z\right)y\right)-u\left(  x_2\right)\right]\lbrack \eta_{m,z}\left(  x_1+G\left( z\right)y\right)-\eta_{m,z}\left(  x_1\right)\\
	&&\quad\quad-\eta_{m,z}\left(  x_2+G\left( z\right)y\right)+\eta_{m,z}\left(  x_2\right)\rbrack\nu\left( dy\right)\vert\\
	&+& \vert\int_{\left\vert y\right\vert> 1}\left[ u\left(  x_2+G\left( z\right)y\right)-u\left(  x_2\right)\right]\lbrack \eta_{m,z}\left(  x_1+G\left( z\right)y\right)-\eta_{m,z}\left(  x_1\right)\\
	&&\quad\quad-\eta_{m,z}\left(  x_2+G\left( z\right)y\right)+\eta_{m,z}\left(  x_2\right)\rbrack\nu\left( dy\right)\vert\\
	&+& \vert\int_{\left\vert y\right\vert> 1}\left[ u\left(  x_1+G\left( z\right)y\right)-u\left(  x_1\right)- u\left(  x_2+G\left( z\right)y\right)+u\left(  x_2\right)\right]\\
	&&\quad\quad\cdot\left[ \eta_{m,z}\left(  x_1+G\left( z\right)y\right)-\eta_{m,z}\left(  x_1\right)\right]\nu\left( dy\right)\vert\\
	&:=& I_1+I_2+I_3+I_4.
\end{eqnarray*}
Then $\eqref{mod1}$ implies that
\begin{eqnarray*}
	I_1, I_2 &\leq& C l\left(m\right)^{1+\beta}w\left( a\right)^{\beta}\left( \varepsilon \left\vert u\right\vert_{1+\beta,\infty}+C_{\varepsilon}\left\vert u\right\vert_{0}\right),
\end{eqnarray*}
where $C$ depends on $K$. Meanwhile,
\begin{eqnarray*}
	I_{3} &\leq& C\left\vert u\right\vert_0\int_{\left\vert y\right\vert> 1}\vert\eta_{m,z}\left(  x_1+G\left( z\right)y\right)-\eta_{m,z}\left(  x_2+G\left( z\right)y\right)-\eta_{m,z}\left(  x_1\right)\\
	&&+\eta_{m,z}\left( x_2\right)\vert \nu\left( dy\right)
	\leq  C\left\vert u\right\vert_0l\left(m\right)^{\beta}w\left( a\right)^{\beta},
\end{eqnarray*}
and obviously,
\begin{eqnarray*}
	I_{4} &\leq&  Cw\left( a\right)^{\beta}\left\vert u\right\vert_{\beta,\infty}\leq  Cw\left( a\right)^{\beta}\left( \varepsilon \left\vert u\right\vert_{1+\beta,\infty}+C_{\varepsilon}\left\vert u\right\vert_{0}\right).
\end{eqnarray*}

Summarizing, 
\begin{eqnarray*}
	\sup_{z}\left[\langle u,\eta_{m,z}\rangle_{z}\right]_{\beta}\leq Cl\left(m\right)^{1+\beta}\left( \varepsilon \left\vert u\right\vert_{1+\beta,\infty}+C_{\varepsilon}\left\vert u\right\vert_{0}\right),
\end{eqnarray*}
and thus,
\begin{eqnarray*}
\sup_{z}\left\vert \langle u,\eta_{m,z}\rangle_{z}\right\vert_{\beta,\infty}
	&\leq&  C l\left(m\right)^{1+\beta}\left( \varepsilon \left\vert u\right\vert_{1+\beta,\infty}+C_{\varepsilon}\left\vert u\right\vert_{0}\right).
\end{eqnarray*}
\end{proof}

\subsection{Lower Order Operators}
For any function $u\in  C_b^{2}\left(H_T\right)$, denote
\begin{eqnarray*}
	&&Q_{t,z,z'}u\left( t,x\right)\\
	&:=& 1_{\alpha\in\left( 1,2\right)}b\left( t,z\right)\cdot\nabla u\left(t, x\right)+p\left( t,z\right)u\left( t,x\right) +\int_{ \mathbf{R}^d_0}\lbrack u\left( t,x+q\left( t,z',y\right)\right)\\
	&&-u\left( t,x\right)-\nabla u\left( t,x\right)\cdot q\left( t,z',y\right)1_{\alpha\in\left( 1,2\right)}1_{\left\vert y\right\vert\leq 1}\rbrack\varrho\left(t,z,y\right) \nu_2\left( dy\right)\\
	&:=& 1_{\alpha\in\left( 1,2\right)}b\left( t,z\right)\cdot\nabla u\left(t, x\right)+p\left( t,z\right)u\left( t,x\right) +\tilde{Q}_{t,z,z'}u\left( t,x\right).
\end{eqnarray*}

\begin{lemma}\label{bb}
	Let \textbf{B($K,\beta$)} hold. $\beta\in\left( 0,1/\alpha\right)$. Then for any $u\in \tilde{C}_{\infty ,\infty}^{1+\beta}\left(H_T\right)$ and any $\varepsilon\in\left( 0,1\right)$, there exists $\beta'\in\left(0,\beta\right)$,
	\begin{eqnarray}
	\sup_{t,z,z'}\left\vert \tilde{Q}_{t,z,z'}u\left( t,\cdot\right)\right\vert_{0}&\leq&  C\sup_{t,z,y}\left\vert \rho\left( t, z,y\right)\right\vert \left\vert u\right\vert_{1+\beta',\infty},\\
	\sup_{t,z,z'}\left[ \tilde{Q}_{t,z,z'}u\left( t,\cdot\right)\right]_{\beta}&\leq& \sup_{t,z,y}\left\vert \rho\left( t, z,y\right)\right\vert\left(\varepsilon \left\vert u\right\vert_{1+\beta,\infty}+C_{\varepsilon}\left\vert u\right\vert_{0}\right),
	\end{eqnarray}
	where $C,C_{\varepsilon}$ are independent of $u$.
\end{lemma}
\begin{proof}
	We split the integral.
	\begin{eqnarray*}
		&&\left\vert \tilde{Q}_{t,z,z'}u\left( t,x\right)\right\vert_{0}\\
		&\leq& C\sup_{t,z,y}\left\vert \rho\left( t, z,y\right)\right\vert\int_{ \left\vert y\right\vert\leq 1}1_{\alpha\in\left( 1,2\right)}\left\vert \nabla^{\alpha}u\left( t,x;q\left( t,z',y\right)\right)\right\vert \nu_2\left( dy\right)\\
		&+& C\sup_{t,z,y}\left\vert \rho\left( t, z,y\right)\right\vert\int_{ \mathbf{R}^d_0}1_{\alpha\in\left( 0,1\right]}\left\vert u\left( t,x+q\left( t,z',y\right)\right)-u\left( t,x\right)\right\vert \nu_2\left( dy\right)\\
		&+& C\sup_{t,z,y}\left\vert \rho\left( t, z,y\right)\right\vert\int_{ \left\vert y\right\vert> 1}1_{\alpha\in\left( 1,2\right)}\left\vert u\left( t,x+q\left( t,z',y\right)\right)-u\left( t,x\right)\right\vert \nu_2\left( dy\right)\\
		&:=& C\sup_{t,z,y}\left\vert \rho\left( t, z,y\right)\right\vert\left(I_1+I_2+I_3\right).
	\end{eqnarray*}
	
	Use assumptions \textbf{$\mathbf{\tilde{A}}(w,l,\gamma)$} and \textbf{B($K,\beta$)}. By Lemma \ref{diff},
	\begin{eqnarray*}
		I_1 &\leq& C\left\vert L^{\mu}u\right\vert_0\int_{ \left\vert y\right\vert\leq 1}w\left( \left\vert q\left( t,z',y\right)\right\vert\right)\nu_2\left( dy\right)\leq C\left\vert L^{\mu}u\right\vert_0.
	\end{eqnarray*}
	Take $\beta'\in\left( 0,\delta\right)$. Then by Lemma \ref{dif},
	\begin{eqnarray*}
		I_2 &\leq& C1_{\alpha\in\left( 0,1\right)}\int_{ \mathbf{R}^d_0}\left( \left\vert L^{\mu,1+\beta'}u\right\vert_0 w\left( \left\vert q\left( t,z',y\right)\right\vert\right)^{1+\beta'}\wedge \left\vert u\right\vert_0\right)\nu_2\left( dy\right)\\
		&&+ C1_{\alpha=1}\int_{ \mathbf{R}^d_0}\left( \left\vert \nabla u\right\vert_0  \left\vert q\left( t,z',y\right)\right\vert\wedge \left\vert u\right\vert_0\right)\nu_2\left( dy\right)\\
		&\leq& C\left(1_{\alpha\in\left( 0,1\right)}\left\vert L^{\mu,1+\beta'}u\right\vert_0+1_{\alpha=1}\left\vert \nabla u\right\vert_0+\left\vert u\right\vert_0\right).
	\end{eqnarray*}
	Clearly, $I_3 \leq C\left(\left\vert \nabla u\right\vert_0+\left\vert u\right\vert_0\right)$. Summarizing, there exists $\beta'\in\left( 0,\beta\right)$ so that
	\begin{eqnarray*}
		\sup_{t,z,z'}\left\vert \tilde{Q}_{t,z,z'}u\left( t,x\right)\right\vert_{0}\leq  C \sup_{t,z,y}\left\vert \rho\left( t, z,y\right)\right\vert\left\vert u\right\vert_{1+\beta',\infty}.
	\end{eqnarray*}
	
	Meanwhile,
	\begin{eqnarray*}
		&&\sup_{t,z,y}\left\vert \rho\left( t, z,y\right)\right\vert^{-1}\left\vert \tilde{Q}_{t,z,z'}u\left( t,x_1\right)-\tilde{Q}_{t,z,z'}u\left( t,x_2\right)\right\vert\\
		&\leq& C\int_{ \left\vert y\right\vert\leq 1}1_{\alpha\in\left( 1,2\right)}\left\vert \nabla^{\alpha}u\left( t,x_1;q\left( t,z',y\right)\right)-\nabla^{\alpha}u\left( t,x_2;q\left( t,z',y\right)\right)\right\vert \nu_2\left( dy\right)\\
		&+& C\int_{\left\vert y\right\vert> 1}1_{\alpha\in\left( 1,2\right)}\vert u\left( t,x_1+q\left( t,z',y\right)\right)-u\left( t,x_1\right)-u\left( t,x_2+q\left( t,z',y\right)\right)\\
		&&+u\left( t,x_2\right)\vert \nu_2\left( dy\right)\\
		&+& C\int_{  \mathbf{R}^d_0}1_{\alpha\in\left( 0,1\right]}\vert u\left( t,x_1+q\left( t,z',y\right)\right)-u\left( t,x_1\right)-u\left( t,x_2+q\left( t,z',y\right)\right)\\
		&&+u\left( t,x_2\right)\vert \nu_2\left( dy\right)\\
		&:=&C\left(I_4+I_5+I_6\right).
	\end{eqnarray*}
	
	Set $\left\vert x_1-x_2\right\vert=a$. Then for any $\epsilon\in\left( 0,1\right)$,
	\begin{eqnarray*}
		I_4&=&\int_{ \left\vert q\left( t,z',y\right)\right\vert\leq \epsilon,\left\vert y\right\vert\leq 1}\left\vert \nabla^{\alpha}u\left( t,x_1;q\left( t,z',y\right)\right)-\nabla^{\alpha}u\left( t,x_2;q\left( t,z',y\right)\right)\right\vert \nu_2\left( dy\right)\\
		&+& \int_{ \left\vert q\left( t,z',y\right)\right\vert> \epsilon,\left\vert y\right\vert\leq 1}\left\vert \nabla^{\alpha}u\left( t,x_1;q\left( t,z',y\right)\right)-\nabla^{\alpha}u\left( t,x_2;q\left( t,z',y\right)\right)\right\vert \nu_2\left( dy\right)\\
		&:=&I_{41}+I_{42}.
	\end{eqnarray*}
	By \textbf{B($K,\beta$)}, for any $\varepsilon\in\left( 0,1\right)$, there exists $\epsilon\in\left( 0,1\right)$ such that
	\begin{eqnarray*}
		I_{41} &=&w\left(a\right)^{\beta}\left[L^{\mu}u\left( t,\cdot\right)\right]_{\beta}\int_{ \left\vert q\left( t,z',y\right)\right\vert\leq \epsilon,\left\vert y\right\vert\leq 1}w\left(\left\vert q\left( t,z',y\right)\right\vert\right) \nu_2\left( dy\right)\\
		&\leq& Cw\left(a\right)^{\beta}\left\vert u\left( t,\cdot\right)\right\vert_{1+\beta,\infty}\int_{ \left\vert q\left( t,z',y\right)\right\vert\leq \epsilon}w\left(\left\vert q\left( t,z',y\right)\right\vert\right) \nu_2\left( dy\right)\\
		&\leq& \frac{\varepsilon}{5}w\left(a\right)^{\beta}\left\vert u\left( t,\cdot\right)\right\vert_{1+\beta,\infty}.
	\end{eqnarray*}
	There also exists $\kappa\in\left( 0,1\right)$ so that 
	\begin{eqnarray*}
		I_{42} &=& w\left(a\right)^{\beta}\left[L^{\mu,\kappa}u\left( t,\cdot\right)\right]_{\beta}\int_{ \left\vert q\left( t,z',y\right)\right\vert> \epsilon,\left\vert y\right\vert\leq 1}w\left(\left\vert q\left( t,z',y\right)\right\vert\right)^{\kappa} \nu_2\left( dy\right)\\
		&\leq& Cl\left(\epsilon^{-1}\right)^{1-\kappa}w\left(a\right)^{\beta}\left\vert u\left( t,\cdot\right)\right\vert_{\kappa+\beta,\infty}\int_{\left\vert y\right\vert\leq 1}w\left(\left\vert q\left( t,z',y\right)\right\vert\right) \nu_2\left( dy\right)\\
		&\leq& Cl\left(\epsilon^{-1}\right)^{1-\kappa}w\left(a\right)^{\beta}\left\vert u\left( t,\cdot\right)\right\vert_{\kappa+\beta,\infty}.
	\end{eqnarray*}
	Applying \cite[Proposition 4]{mf}, we can always attain
	\begin{eqnarray*}
		I_{42} &\leq& w\left(a\right)^{\beta}\left(\frac{\varepsilon}{5} \left\vert u\right\vert_{1+\beta,\infty}+C_{\varepsilon}\left\vert u\right\vert_{0}\right),
	\end{eqnarray*}
	which concludes $I_{4} \leq w\left(a\right)^{\beta}\left(\frac{2\varepsilon}{5} \left\vert u\right\vert_{1+\beta,\infty}+C_{\varepsilon}\left\vert u\right\vert_{0}\right)$.
	
	In the mean time, by \textbf{B($K,\beta$)}, Lemma \ref{deri} and \cite[Proposition 4]{mf},
	\begin{eqnarray*}
		I_5 &\leq& 1_{\alpha\in\left( 1,2\right)}w\left( a\right)^{\beta}\int_{\left\vert y\right\vert> 1}\left[\nabla u\left( t,\cdot\right)\right]_{\beta}\left\vert q\left( t,z',y\right)\right\vert \wedge \left[ u\left( t,\cdot\right)\right]_{\beta}\nu_2\left( dy\right)\\
		&\leq& C1_{\alpha\in\left( 1,2\right)}w\left(a\right)^{\beta}\left(\frac{\varepsilon}{5} \left\vert u\right\vert_{1+\beta,\infty}+C_{\varepsilon}\left\vert u\right\vert_{0}\right).
	\end{eqnarray*}
	
	Besides, for any $\epsilon\in\left( 0,1\right)$,
	\begin{eqnarray*}
		I_6&\leq& \int_{  \left\vert q\left( t,z',y\right)\right\vert\leq \epsilon}1_{\alpha\in\left( 0,1\right]}\vert u\left( t,x_1+q\left( t,z',y\right)\right)-u\left( t,x_1\right)\\
		&&-u\left( t,x_2+q\left( t,z',y\right)\right)+u\left( t,x_2\right)\vert \nu_2\left( dy\right)\\
		&+& \int_{ \left\vert q\left( t,z',y\right)\right\vert> \epsilon}1_{\alpha\in\left( 0,1\right]}\vert u\left( t,x_1+q\left( t,z',y\right)\right)-u\left( t,x_1\right)\\
		&&-u\left( t,x_2+q\left( t,z',y\right)\right)+u\left( t,x_2\right)\vert \nu_2\left( dy\right)\\
		&:=&I_{61}+I_{62}.
	\end{eqnarray*}
	We first discuss the case $\alpha\in\left( 0,1\right)$. Applying Lemma \ref{dif}, we have
	\begin{eqnarray*}
		I_{61} &\leq& C\left[ L^{\mu}u\left( t,\cdot\right)\right]_{\beta}w\left( a\right)^{\beta}\int_{  \left\vert q\left( t,z',y\right)\right\vert\leq \epsilon}1_{\alpha\in\left( 0,1\right]} w\left( \left\vert q\left( t,z',y\right)\right\vert\right)\nu_2\left( dy\right)\\
		&\leq& Cw\left( a\right)^{\beta}\left\vert u\left( t,\cdot\right)\right\vert_{1+\beta,\infty }\int_{  \left\vert q\left( t,z',y\right)\right\vert\leq \epsilon}1_{\alpha\in\left( 0,1\right]} w\left( \left\vert q\left( t,z',y\right)\right\vert\right) \nu_2\left( dy\right).
	\end{eqnarray*}
	On the other hand,
	\begin{eqnarray*}
		I_{62} &\leq& Cw\left( a\right)^{\beta}\int_{  \left\vert q\left( t,z',y\right)\right\vert> \epsilon}\left[ L^{\mu,\frac{1}{2}}u\left( t,\cdot\right)\right]_{\beta} w\left( \left\vert q\left( t,z',y\right)\right\vert\right)^{\frac{1}{2}}\wedge \left[ u\right]_{\beta}\nu_2\left( dy\right)\\
		&\leq& Cw\left( a\right)^{\beta}l\left( \epsilon^{-1}\right)^{\frac{1}{2}}\left\vert u\left( t,\cdot\right)\right\vert_{\frac{1}{2}+\beta,\infty }\int_{  \left\vert q\left( t,z',y\right)\right\vert> \epsilon}\left( w\left(\left\vert q\left( t,z',y\right)\right\vert\right)\wedge 1\right) \nu_2\left( dy\right).
	\end{eqnarray*}
	As what we did for $I_4$, by choosing an appropriate $\epsilon$, we have 
	\begin{eqnarray*}
		I_{6} &\leq& w\left(a\right)^{\beta}\left(\frac{\varepsilon}{5} \left\vert u\right\vert_{1+\beta,\infty}+C_{\varepsilon}\left\vert u\right\vert_{0}\right),
	\end{eqnarray*}
	If $\alpha=1$, then 
	\begin{eqnarray*}
		I_{61} &\leq& C\left[ \nabla u\left( t,\cdot\right)\right]_{\beta'}w\left( a\right)^{\beta'}\int_{  \left\vert q\left( t,z',y\right)\right\vert\leq \epsilon}  \left\vert q\left( t,z',y\right)\right\vert\nu_2\left( dy\right)\\
		&\leq& Cw\left( a\right)^{\beta'}\left\vert u\left( t,\cdot\right)\right\vert_{1+\beta,\infty }\int_{  \left\vert q\left( t,z',y\right)\right\vert\leq \epsilon}\left\vert  q\left( t,z',y\right)\right\vert \nu_2\left( dy\right)
	\end{eqnarray*}
	for all $\beta'\in\left( 0,\beta\right)$. Examining the proof of Lemma \ref{deri}, we find that this constant $C$ is uniformly bounded under \textbf{$\mathbf{\tilde{A}}(w,l,\gamma)$}(ii) for all $\beta'\in\left( 0,\beta\right)$. Thus,
	\begin{eqnarray*}
		I_{61} &\leq& Cw\left( a\right)^{\beta}\left\vert u\left( t,\cdot\right)\right\vert_{1+\beta,\infty }\int_{  \left\vert q\left( t,z',y\right)\right\vert\leq \epsilon}\left\vert  q\left( t,z',y\right)\right\vert \nu_2\left( dy\right).
	\end{eqnarray*}
	Estimate $I_{62}$ in the same way as above. By choosing an appropriate $\epsilon$, we arrive at 
	\begin{eqnarray*}
		I_{6} &\leq& w\left(a\right)^{\beta}\left(\frac{\varepsilon}{5} \left\vert u\right\vert_{1+\beta,\infty}+C_{\varepsilon}\left\vert u\right\vert_{0}\right).
	\end{eqnarray*}

	As a summary, for all $\alpha\in\left( 0,2\right)$ and any $\varepsilon\in\left( 0,1\right)$, there exists $C_{\varepsilon}$ that is independent of $u$ so that
	\begin{eqnarray*}
		\sup_{t,z,z'}\left[ \tilde{Q}_{t,z,z'}u\left( t,\cdot\right)\right]_{\beta}&\leq&  \sup_{t,z,y}\left\vert \varrho\left( t,z,y\right)\right\vert\left(\varepsilon \left\vert u\right\vert_{1+\beta,\infty}+C_{\varepsilon}\left\vert u\right\vert_{0}\right).
	\end{eqnarray*}
\end{proof}

\begin{lemma}\label{cc}
	Let \textbf{B($K,\beta$)} hold. $\beta\in\left( 0,1/\alpha\right)$. Then for any $u\in \tilde{C}_{\infty ,\infty}^{1+\beta}\left(H_T\right)$ and any $\varepsilon\in\left( 0,1\right)$, there exists $C_{\varepsilon}$ independent of $u$ such that 
	\begin{eqnarray}\label{c1}
	\left\vert \mathcal{Q}u\left( t,\cdot\right)\right\vert_{\beta,\infty}&\leq&  \varepsilon \left\vert u\left( t,\cdot\right)\right\vert_{1+\beta,\infty}+C_{\varepsilon}\left\vert u\left( t,\cdot\right)\right\vert_{0}.
	\end{eqnarray}
\end{lemma}
\begin{proof}
	Note that $\mathcal{Q}u\left( t,x\right)=Q_{t,x,x}u\left( t,x\right)$. By Lemmas \ref{bb} and \ref{deri}
	\begin{eqnarray*}
		\left\vert \mathcal{Q}u\left( t,\cdot\right)\right\vert_{0}&\leq&  C \left\vert u\left( t,\cdot\right)\right\vert_{1+\beta',\infty}
	\end{eqnarray*}
	for some $\beta'\in\left( 0,\beta\right)$. Meanwhile, for any $x,h\in\mathbf{R}^d$,
	\begin{eqnarray*}
		&&\left\vert Q_{t,x+h,x+h}u\left( t,x+h\right)-Q_{t,x,x}u\left( t,x\right)\right\vert\\
		&\leq& 1_{\alpha\in\left( 1,2\right)}\left\vert b\left( t,x+h\right)\nabla u\left( t,x+h\right)-b\left( t,x\right)\nabla u\left( t,x\right)\right\vert \\
		&+&\left\vert p\left( t,x+h\right)u\left( t,x+h\right)-p\left( t,x\right)u\left( t,x\right)\right\vert\\
		&+& \left\vert \tilde{Q}_{t,x+h,x+h}u\left( t,x+h\right)-\tilde{Q}_{t,x,x}u\left( t,x\right)\right\vert.
	\end{eqnarray*}
	
	Obviously,
	\begin{eqnarray*}
		&& 1_{\alpha\in\left( 1,2\right)}\left\vert b\left( t,x+h\right)\nabla u\left( t,x+h\right)-b\left( t,x\right)\nabla u\left( t,x\right)\right\vert \\
		&+&\left\vert p\left( t,x+h\right)u\left( t,x+h\right)-p\left( t,x\right)u\left( t,x\right)\right\vert\\
		&\leq& Cw\left( \left\vert h\right\vert\right)^{\beta}\left(\left\vert u\left( t,\cdot\right)\right\vert_{\beta}+1_{\alpha\in\left( 1,2\right)}\left\vert \nabla u\left( t,\cdot\right)\right\vert_{\beta}\right).
	\end{eqnarray*}
	In the mean time, 
	\begin{eqnarray*}
		&&\left\vert \tilde{Q}_{t,x+h,x+h}u\left( t,x+h\right)-\tilde{Q}_{t,x,x}u\left( t,x\right)\right\vert\\
		&\leq& \left\vert \tilde{Q}_{t,x+h,x+h}u\left( t,x+h\right)-\tilde{Q}_{t,x,x+h}u\left( t,x+h\right)\right\vert\\
		&+& \left\vert \tilde{Q}_{t,x,x+h}u\left( t,x+h\right)-\tilde{Q}_{t,x,x}u\left( t,x+h\right)\right\vert\\
		&+& \left\vert \tilde{Q}_{t,x,x}u\left( t,x+h\right)-\tilde{Q}_{t,x,x}u\left( t,x\right)\right\vert.
	\end{eqnarray*}
	By Lemma \ref{bb},
	\begin{eqnarray*}
		&&\left\vert \tilde{Q}_{t,x+h,x+h}u\left( t,x+h\right)-\tilde{Q}_{t,x,x+h}u\left( t,x+h\right)\right\vert\\
		&\leq&C\sup_{t,y}\left[ \varrho\left( t,\cdot,y\right)\right]_{\beta} w\left( \left\vert h\right\vert\right)^{\beta}\left\vert u\right\vert_{1+\beta',\infty}
	\end{eqnarray*}
	for some $\beta'\in\left( 0,\beta\right)$, and
	\begin{eqnarray*}
		&&\left\vert \tilde{Q}_{t,x,x}u\left( t,x+h\right)-\tilde{Q}_{t,x,x}u\left( t,x\right)\right\vert\\
		&\leq&C\sup_{t,z,y}\left\vert \varrho\left( t,z,y\right)\right\vert w\left( \left\vert h\right\vert\right)^{\beta}\left( \varepsilon \left\vert u\right\vert_{1+\beta,\infty}+C_{\varepsilon}\left\vert u\right\vert_{0}\right).
	\end{eqnarray*}
	Besides, 
	\begin{eqnarray*}
		&&\left\vert \tilde{Q}_{t,x,x+h}u\left( t,x+h\right)-\tilde{Q}_{t,x,x}u\left( t,x+h\right)\right\vert\\
		&\leq& C\int_{ \left\vert y\right\vert\leq 1}1_{\alpha\in\left( 1,2\right)}\left\vert \nabla^{\alpha}u\left( t,x+h;q\left( t,x+h,y\right)\right)-\nabla^{\alpha}u\left( t,x+h;q\left( t,x,y\right)\right)\right\vert \nu_2\left( dy\right)\\
		&+& C\int_{\left\vert y\right\vert> 1}1_{\alpha\in\left( 1,2\right)}\left\vert u\left( t,x+h+q\left( t,x+h,y\right)\right)-u\left( t,x+h+q\left( t,x,y\right)\right)\right\vert \nu_2\left( dy\right)\\
		&+& C\int_{  \mathbf{R}^d_0}1_{\alpha\in\left( 0,1\right]}\left\vert u\left( t,x+h+q\left( t,x+h,y\right)\right)-u\left( t,x+h+q\left( t,x,y\right)\right)\right\vert \nu_2\left( dy\right)\\
		&:=&C\left(I_1+I_2+I_3\right).
	\end{eqnarray*}
	Similarly as what we did in Lemma \ref{bb}, 
	\begin{eqnarray*}
		I_1 &\leq& C\int_{\left\vert y\right\vert\leq 1}\int_{0}^1\left\vert \nabla u\left( t,x+h+\theta q\left( t,x+h,y\right)\right)-\nabla u\left( t,x+h\right)\right\vert d\theta\\
		&&\left\vert q\left( t,x+h,y\right)-q\left( t,x,y\right)\right\vert \nu_2\left( dy\right)\\
		&+& C\int_{\left\vert y\right\vert\leq 1}\int_{0}^1\vert \nabla u\left( t,x+h+\theta q\left( t,x+h,y\right)\right)\\
		&&-\nabla u\left( t,x+h+\theta q\left( t,x,y\right)\right)\vert d\theta\left\vert q\left( t,x,y\right)\right\vert \nu_2\left( dy\right)\\
		&\leq& C\left[\nabla u\left( t,\cdot\right)\right]_{\beta}\int_{\left\vert y\right\vert\leq 1}w\left(  \left\vert q\left( t,x+h,y\right)\right\vert\right)^{\beta}\left\vert q\left( t,x+h,y\right)-q\left( t,x,y\right)\right\vert\nu_2\left( dy\right)\\
		&+& C\left[\nabla u\left( t,\cdot\right)\right]_{\beta}\int_{\left\vert y\right\vert\leq 1}w\left(  \left\vert q\left( t,x+h,y\right)-q\left( t,x,y\right)\right\vert\right)^{\beta}\left\vert q\left( t,x,y\right)\right\vert\nu_2\left( dy\right)\\
		&\leq& C \left\vert u\left( t,\cdot\right)\right\vert_{\kappa+\beta,\infty} w\left( \left\vert h\right\vert\right)^{\beta}
	\end{eqnarray*}
	for some $\kappa\in\left( 0,1\right)$.
	\begin{eqnarray*}
		I_2 &\leq& C\int_{\left\vert y\right\vert> 1}\left\vert\nabla u\left( t,\cdot\right) \right\vert_{0}\left\vert q\left( t,z'+h,y\right)-q\left( t,z',y\right)\right\vert\wedge \left\vert u \right\vert_{0} \nu_2\left( dy\right)\\
		&\leq& C\left\vert u\left( t,\cdot\right)  \right\vert_{1+\beta',\infty} w\left( \left\vert h\right\vert\right)^{\beta}
	\end{eqnarray*}
	for some $\beta'\in\left( 0,\beta\right)$. And
	\begin{eqnarray*}
		I_3 &\leq& C\int_{\mathbf{R}^d_0}1_{\alpha\in\left( 0,1\right)}\left(\left\vert L^{\mu} u\left( t,\cdot\right)  \right\vert_{0}w\left( \left\vert q\left( t,z'+h,y\right)-q\left( t,z',y\right)\right\vert\right)\wedge \left\vert u \right\vert_{0}\right)\nu_2\left( dy\right)\\
		&&+C\int_{\mathbf{R}^d_0}1_{\alpha=1}\left(\left\vert \nabla u\left( t,\cdot\right)  \right\vert_{0} \left\vert q\left( t,z'+h,y\right)-q\left( t,z',y\right)\right\vert\wedge \left\vert u \right\vert_{0}\right)\nu_2\left( dy\right)\\
		&\leq& C1_{\alpha\in\left( 0,1\right]} \left\vert u\left( t,\cdot\right)  \right\vert_{1+\beta',\infty}w\left( \left\vert h\right\vert\right)^{\beta}
	\end{eqnarray*}
	for some $\beta'\in\left( 0,\beta\right)$.
	
	Summarizing, we obtain $\eqref{c1}$.
\end{proof}

\section{Proof of the Main Result}
\subsection{Auxiliary Results} In this section, we state or prove well-posedness for integro-differential equations with space-independent operators. 

\begin{theorem}\cite[Theorem 1.1]{mf}\label{thm2}
	Let $\beta\in\left(0,\infty\right), \lambda\geq 0$ and $\nu$ be a L\'{e}vy measure satisfying \textbf{A(w,l)}. If $f\left( t,x\right)\in  \tilde{C}^{\beta}_{\infty,\infty}\left(H_T\right)$. Then there is a unique solution $u\in\left( t,x\right)\in \tilde{C}^{1+\beta}_{\infty,\infty}\left(H_T\right)$ to
	\begin{eqnarray}\label{eq}
	\partial_t u\left( t,x\right)&=&Lu\left( t,x\right)-\lambda u\left( t,x\right)+f\left( t,x\right), \lambda\geq 0,\\
	u\left( 0,x\right)&=& 0,\quad\left( t,x\right)\in H_T,\nonumber
	\end{eqnarray}
	where for any function $\varphi\in  C_b^{2}\left(\mathbf{R}^d\right)$,
	\begin{eqnarray}
	L\varphi\left( x\right):=\int\left[ \varphi\left( x+y\right)-\varphi\left( x\right)-\chi_{\alpha}\left( y\right)y\cdot \nabla \varphi\left( x\right)\right]\nu\left(d y\right).
	\end{eqnarray}
	Moreover, there exists a constant $C$ depending on $\kappa,\beta, d, T,\mu,\nu$ such that
	\begin{eqnarray}
	\left\vert u\right\vert_{\beta,\infty}&\leq& C\left( \lambda^{-1}\wedge T\right) \left\vert f\right\vert_{\beta,\infty},\label{eest6}\\
	\left\vert u\right\vert_{1+\beta,\infty}&\leq& C\left\vert f\right\vert_{\beta,\infty}\label{eest3}
	\end{eqnarray}
	And there is a constant $C$ depending on $\kappa,\beta,d, T,\mu,\nu$ such that for all $0\leq s<t\leq T$, $\kappa\in\left[ 0,1\right]$, 
	\begin{equation}\label{eest4}
	\left\vert u\left(t,\cdot\right)-u\left(s,\cdot\right)\right\vert_{\kappa+\beta,\infty}\leq C\left\vert t-s\right\vert^{1-\kappa}\left\vert f\right\vert_{\beta,\infty}.
	\end{equation}
\end{theorem}

\begin{theorem}\label{thm1}
	Let $\nu$ be a L\'{e}vy measure, $\alpha\in\left( 0,2\right), \beta\in\left(0,1\right), \lambda\geq 0$. $G$ is an invertible $d\times d$-matrix. Assume that $f\left( t,x\right)\in \tilde{C}^{\beta}_{\infty,\infty}\left(H_T\right)$. Then there is a unique solution $u\in\left( t,x\right)\in \tilde{C}^{1+\beta}_{\infty,\infty}\left(H_T\right)$ to 
	\begin{eqnarray}\label{eeq1}
	\partial_t u\left( t,x\right)&=&Lu\left( t,x\right)-\lambda u\left( t,x\right)+f\left( t,x\right), \lambda\geq 0,\\
	u\left( 0,x\right)&=& 0,\quad\left( t,x\right)\in H_T,\nonumber
	\end{eqnarray}
	 where for any function $\varphi\in  C_b^{2}\left(\mathbf{R}^d\right)$,
	\begin{eqnarray*}
		L\varphi\left( x\right)&:=&\int\left[\varphi\left( x+Gy\right)-\varphi\left( x\right)-\chi_{\alpha}\left( y\right)Gy\cdot \nabla \varphi\left( x\right) \right]\nu\left(d y\right).
	\end{eqnarray*}
	Moreover, there exists a constant $C$ depending on $\kappa,\beta, d, T,\mu,\nu,\left\Vert G^{-1}\right\Vert,\left\Vert G\right\Vert$ such that
	\begin{eqnarray}
	\left\vert u\right\vert_{\beta,\infty}&\leq& C\left( \lambda^{-1}\wedge T\right) \left\vert f\right\vert_{\beta,\infty},\label{est6}\\
	\left\vert u\right\vert_{1+\beta,\infty}&\leq& C\left\vert f\right\vert_{\beta,\infty}\label{est3}
	\end{eqnarray}
	And there is a constant $C$ depending on $\kappa,\beta,d, T,\mu,\nu,\left\Vert G^{-1}\right\Vert,\left\Vert G\right\Vert$ such that for all $0\leq s<t\leq T$, $\kappa\in\left[ 0,1\right]$,
	\begin{equation}\label{est4}
	\left\vert u\left(t,\cdot\right)-u\left(s,\cdot\right)\right\vert_{\kappa+\beta,\infty}\leq C\left\vert t-s\right\vert^{1-\kappa}\left\vert f\right\vert_{\beta,\infty}.
	\end{equation}
\end{theorem}
\begin{proof}
	We first assume $f\left( t,x\right)\in  C_b^{\infty}\left(H_T\right)\cap \tilde{C}^{\beta}_{\infty,\infty}\left(H_T\right)$.
	
	\textsc{Existence. }Denote $F\left(r,Z^{\nu}_r\right)=e^{-\lambda\left( r-s\right)}f\left(s, x+GZ^{\nu}_r-GZ^{\nu}_s\right), s\leq r\leq t,$ and apply the It\^{o} formula to $F\left(r,Z^{\nu}_r\right)$ on $\left[s,t\right]$.
	\begin{eqnarray*}
		&& e^{-\lambda\left( t-s\right)}f\left(s, x+GZ^{\nu}_t-GZ^{\nu}_s\right)-f\left( s,x\right)\\
		&=& -\lambda\int_s^t  F\left(r,Z^{\nu}_{r}\right)dr+\int_s^t\int \chi_{\alpha}\left( y\right)y\cdot\nabla F\left(r, Z^{\nu}_{r-}\right)\tilde{J}\left(dr,dy\right)\\
		&&+ \int_s^t \int \left[ F\left(r,Z^{\nu}_{r-}+y\right)-F\left(r,Z^{\nu}_{r-}\right)-\chi_{\alpha}\left(y\right) y\cdot \nabla F\left(r,Z^{\nu}_{r-}\right)\right]J\left( dr,dy\right).
	\end{eqnarray*}
	Take expectation for both sides.
	\begin{eqnarray*}
		&&  e^{-\lambda\left( t-s\right)}\mathbf{E}f\left(s, x+GZ^{\nu}_t-GZ^{\nu}_s\right)-f\left( s,x\right)\\
		&=& -\lambda\int_s^t  e^{-\lambda\left( r-s\right)}\mathbf{E}f\left(s, x+GZ^{\nu}_r-GZ^{\nu}_s\right)dr\\
		&&+\int_s^t L e^{-\lambda\left( r-s\right)}\mathbf{E}f\left(s, x+GZ^{\nu}_{r-}-GZ^{\nu}_s\right)dr.
	\end{eqnarray*}
	Integrate both sides over $\left[0,t\right]$ with respect to $s$ and obtain
	\begin{eqnarray*}
		&&  \int_0^t e^{-\lambda\left( t-s\right)}\mathbf{E}f\left(s, x+GZ^{\nu}_t-GZ^{\nu}_s\right)ds-\int_0^t f\left( s,x\right)ds\\
		&=& -\lambda\int_0^t\int_0^r  e^{-\lambda\left( r-s\right)}\mathbf{E}f\left(s, x+GZ^{\nu}_r-GZ^{\pi}_s\right)dsdr\\
		&&+\int_0^t L\int_0^r  e^{-\lambda\left( r-s\right)}\mathbf{E}f\left(s, x+GZ^{\nu}_{r-}-GZ^{\nu}_s\right)dsdr,
	\end{eqnarray*}
	which shows $u\left(t,x\right)=\int_0^t e^{-\lambda\left( t-s\right)}\mathbf{E}f\left(s, x+GZ_{t-s}^{\nu}\right)ds$ solves $\eqref{eeq1}$ in the integral sense. As a result of the dominated convergence theorem and Fubini's theorem, $u\in C^{\infty}_b\left( H_T\right)$. And by the equation, $u$ is continuously differentiable in $t$.
	
	\textsc{Uniqueness. }Suppose there are two solutions $u_1,u_2$ solving the equation, then $u:=u_1-u_2$ solves 
	\begin{eqnarray}\label{uni}
	\partial_t u\left( t,x\right)&=&L u\left(t,x\right)-\lambda u\left(t,x\right),\\
	u\left( 0,x\right)&=& 0.\nonumber
	\end{eqnarray}
	
	Apply the It\^{o} formula to $v\left( t-s,Z^{\nu}_s\right):=e^{-\lambda s}u\left(t-s,x+GZ^{\nu}_s\right)$, $0\leq s\leq t,$ over $\left[0,t\right]$ and take expectation for both sides of the resulting identity, then 
	\begin{eqnarray*}
		u\left( t,x\right)=-\mathbf{E}\int_0^t e^{-\lambda s}\left[ \left(-\partial_t u-\lambda u+Lu\right)\circ\left(t-s,x+GZ^{\nu}_{s-}\right)\right]ds=0.
	\end{eqnarray*}
	
	\textsc{Estimates for Smooth Inputs. } Denote $g\left(t,x\right)=f\left(t, Gx\right), x\in\mathbf{R}^d$. Then by Lemma \ref{lem1}, for any $\beta\in\left( 0,1\right)$,
	\begin{eqnarray*}
		\left\vert u\left( t,\cdot\right)\right\vert_{\beta,\infty}&\leq& \left\vert \int_0^t e^{-\lambda\left( t-s\right)}\mathbf{E}f\left(s, G\cdot+GZ_{t-s}^{\nu}\right)ds\right\vert_{\beta,\infty}\\
		&=& \left\vert \int_0^t e^{-\lambda\left( t-s\right)}\mathbf{E}g\left(s, \cdot+Z_{t-s}^{\nu}\right)ds\right\vert_{\beta,\infty}\\
		&\leq& C\left( \lambda^{-1}\wedge T\right) \left\vert g\left( t,\cdot\right)\right\vert_{\beta,\infty}\\
		&\leq& C\left( \lambda^{-1}\wedge T\right) \left\vert f\left( t,\cdot\right)\right\vert_{\beta,\infty}.
	\end{eqnarray*}
	Similarly, we can prove $\eqref{est3},\eqref{est4}$.

\textsc{Estimates for H\"{o}lder Inputs. }This part of proof is identical to section 5 of \cite{mf}.
\end{proof}

\subsection{Proof of Theorem 1.1}
We aim at providing a unifying proof for both $\mathcal{L}=\mathcal{A}+\mathcal{Q}$ and $\mathcal{L}=\mathcal{G}+\mathcal{Q}$. Before that, we claim

\begin{lemma}\label{beta}
	Let $\beta\in\left( 0,1/\alpha\right)$ and $f,g\in\tilde{C}^{\beta}\left( \mathbf{R}^d\right)$. Then 
	\begin{eqnarray*}
		\left\vert fg\right\vert_{\beta} &\leq& \left\vert f\right\vert_{0}\left\vert g\right\vert_{0}+\left\vert f\right\vert_{0}\left[g\right]_{\beta}+\left\vert g\right\vert_{0}\left[f\right]_{\beta},\\
		\left\vert f\right\vert_{0} &=&\sup_{z}\left\vert \eta_{m,z}f\right\vert_{0},\quad\forall k\in\mathbf{N}_{+},
	\end{eqnarray*}
	and for some positive constant $C$ that does not depend on $m$,
	\begin{eqnarray*}
		\left\vert f\right\vert_{\beta}&\leq& C l\left( m\right)^{\beta}\left\vert f\right\vert_{0}+\sup_{z}\left\vert \eta_{m,z}f\right\vert_{\beta},\\
		\sup_{z}\left\vert \eta_{m,z}f\right\vert_{\beta}&\leq&C l\left( m\right)^{\beta}\left\vert f\right\vert_0+\left\vert f\right\vert_{\beta}.
	\end{eqnarray*}
\end{lemma}
\begin{proof}
	Proof for the first two is identical to that for the standard H\"{o}lder norm. By monotonicity of the scaling factor,
	\begin{eqnarray*}
		\left\vert f\right\vert_{\beta}&\leq& C\sup_{\left\vert x-y\right\vert>\frac{1}{m}} l\left( \frac{1}{\left\vert x-y\right\vert}\right)^{\beta}\left\vert f\right\vert_{0}+\sup_{\left\vert x-y\right\vert\leq\frac{1}{m}}\frac{\left\vert f\left( x\right)-f\left(y\right)\right\vert}{w\left(\left\vert x-y\right\vert\right)^{\beta}}\\
		&\leq& C l\left( m\right)^{\beta}\left\vert f\right\vert_{0}+\sup_{z} \frac{\left\vert\eta_{m,z}\left( x\right) f\left( x\right)-\eta_{m,z}\left( y\right)f\left(y\right)\right\vert}{w\left(\left\vert x-y\right\vert\right)^{\beta}}\\
		&\leq& C l\left( m\right)^{\beta}\left\vert f\right\vert_{0}+\sup_{z}\left\vert \eta_{m,z}f\right\vert_{\beta}.
	\end{eqnarray*} 
	Meanwhile,
	\begin{eqnarray}
	\sup_{z}\left\vert \eta_{m,z}f\right\vert_{\beta}&\leq& \left\vert f\right\vert_{0}\left\vert \eta\right\vert_{0}+\left\vert f\right\vert_{0}\sup_{z}\left[\eta_{m,z}\right]_{\beta}+\left\vert \eta\right\vert_{0}\left[f\right]_{\beta}\nonumber\\
	&\leq& Cl\left(m\right)^{\beta}\left\vert f\right\vert_{0}+\left\vert f\right\vert_{\beta}.\label{est}
	\end{eqnarray} 
\end{proof}

Without introducing much confusion, we will use $L_z$ to represent $\eqref{opp3}$ and $\eqref{opp1}$ at the same time, and $\langle u,\eta_{m,z}\rangle_z$ for both $\eqref{opp4}$ and $\eqref{opp2}$. We will also use $\left\vert \cdot\right\vert_{\beta}$ and $\left\vert \cdot\right\vert_{\beta, \infty}$ interchangeably, which is justified by \textbf{$\mathbf{\tilde{A}}(w,l,\gamma)$}(ii). 

\textsc{Estimates and Uniqueness. }Let $u\in\tilde{C}^{1+\beta}_{\infty,\infty}\left(H_T\right)$ be a solution to $\eqref{eq3}$, either $\mathcal{L}=\mathcal{A}+\mathcal{Q}$ or $\mathcal{L}=\mathcal{G}+\mathcal{Q}$. Obviously,
\begin{eqnarray*}
	\partial_t \left( \eta_{m,z} u\right)
	&=& \eta_{m,z}\left( L_z u\right)-\lambda \left(\eta_{m,z} u\right)+\eta_{m,z}f+\eta_{m,z}\left[ \left( \mathcal{L}-L_z \right)u\right],
\end{eqnarray*}
where by elementary derivation,
\begin{eqnarray}\label{same}
\eta_{m,z}\left( L_z u\right)=L_z\left(\eta_{m,z} u\right)-u\left(L_z\eta_{m,z} \right)-\langle u,\eta_{m,z}\rangle_z,
\end{eqnarray}
therefore, $\eta_{m,z} u$ solves
\begin{eqnarray}
\partial_t \left( \eta_{m,z} u\right)&=&L_z\left(\eta_{m,z} u\right)-\lambda \left(\eta_{m,z} u\right)-u\left(L_z\eta_{m,z} \right)+\eta_{m,z}f\nonumber\\
&& +\eta_{m,z}\left[ \left( \mathcal{L}-L_z \right)u\right]-\langle u,\eta_{m,z}\rangle_z.\label{g7}
\end{eqnarray}

As an application of Theorems \ref{thm2} and \ref{thm1}, we have 
\begin{eqnarray}
\left\vert  \eta_{m,z} u\right\vert_{1+\beta,\infty}&\leq& C\big( \left\vert u\left(L_z\eta_{m,z} \right)\right\vert_{\beta,\infty}+\left\vert\eta_{m,z}f\right\vert_{\beta,\infty}\label{g1}\\
&& +\left\vert\eta_{m,z}\left[ \left( \mathcal{L}-L_z \right)u\right]\right\vert_{\beta,\infty}+\left\vert \langle u,\eta_{m,z}\rangle_z\right\vert_{\beta,\infty}\big),\nonumber
\end{eqnarray}
\begin{eqnarray*}
	\left\vert \eta_{m,z} u\right\vert_{\beta,\infty}&\leq& C\left( \lambda^{-1}\wedge T\right)\big( \left\vert u\left(L_z\eta_{m,z} \right)\right\vert_{\beta,\infty}+\left\vert\eta_{m,z}f\right\vert_{\beta,\infty}\\
	&& +\left\vert\eta_{m,z}\left[ \left( \mathcal{L}-L_z \right)u\right]\right\vert_{\beta,\infty}+\left\vert \langle u,\eta_{m,z}\rangle_z\right\vert_{\beta,\infty}\big)
\end{eqnarray*}
for some $C$ independent of $\lambda$. Clearly, $\left\vert L_z\eta_{m,z} \right\vert_{\beta,\infty}\leq C\left\vert \eta_{m,z} \right\vert_{1+\beta,\infty}
\leq C l\left(m\right)^{1+\beta}$. Then by Lemma \ref{beta} and \cite[Proposition 4]{mf},
\begin{eqnarray*}
	&&\left\vert u\left(L_z\eta_{m,z} \right)\right\vert_{\beta,\infty}\leq C \left\vert u\left(L_z\eta_{m,z} \right)\right\vert_{\beta}\\
	&\leq& \left\vert u\right\vert_{\beta}\left\vert L_z\eta_{m,z} \right\vert_{0}+\left\vert u\right\vert_{0}\left\vert L_z\eta_{m,z} \right\vert_{\beta}\\
	&\leq& Cl\left(m\right)^{1+\beta}\left\vert u\right\vert_{\beta,\infty}\leq Cl\left(m\right)^{1+\beta}\left( \varepsilon \left\vert u\right\vert_{1+\beta}+C_{\varepsilon}\left\vert u\right\vert_{0}\right).
\end{eqnarray*}
Apply Lemma \ref{beta} again.
\begin{eqnarray*}
	&&\left\vert\eta_{m,z}\left[ \left( \mathcal{L}-L_z \right)u\right]\right\vert_{\beta,\infty}\\
	&\leq& C\left(l\left(m\right)^{\beta}\left\vert \left( \widetilde{\mathcal{L}}-L_z \right)u\right\vert_{0}+l\left(m\right)^{\beta}\left\vert  \mathcal{Q}u\right\vert_{0}+\left\vert \left( \widetilde{\mathcal{L}}-L_z \right)u\right\vert_{\beta,\infty}+\left\vert \mathcal{Q}u\right\vert_{\beta,\infty}\right),
\end{eqnarray*}
where $\widetilde{\mathcal{L}}$ is either $\mathcal{A}$ or $\mathcal{G}$. Then by Lemmas \ref{cont2}, \ref{prr1}, \ref{cc} and Corollaries \ref{coo},\ref{col},
\begin{eqnarray*}
	&&\left\vert\eta_{m,z}\left[ \left( \mathcal{L}-L_z \right)u\right]\right\vert_{\beta,\infty}\\
	&\leq& Cl\left(m\right)^{\beta}\left( \varepsilon \left\vert u\right\vert_{1+\beta,\infty}+C_{\varepsilon}\left\vert u\right\vert_{0}\right)+CF\left(m,x,z\right)\left\vert u \right\vert_{1+\beta,\infty},
\end{eqnarray*}
where 
\begin{eqnarray*}
F\left(m,x,z\right):=\sup_{t,y,\left\vert x-z\right\vert\leq 2/m}\left\vert \rho\left(t,x,y\right)-\rho\left(t,z,y\right)\right\vert_0	
\end{eqnarray*}
if $\mathcal{L}=\mathcal{A}+\mathcal{Q}$, and $F\left(m,x,z\right):=\sup_{\left\vert x-z\right\vert\leq 2/m}\left\Vert G\left( x\right)-G\left( z\right)\right\Vert^{\sigma}$ if $\mathcal{L}=\mathcal{G}+\mathcal{Q}$.

Combining Lemmas \ref{aa}, \ref{aa2}, we obtain
\begin{eqnarray}
&&\left\vert u\left(L_z\eta_{m,z} \right)\right\vert_{\beta,\infty}+\left\vert\eta_{m,z}f\right\vert_{\beta,\infty}\notag\\
&&+\left\vert\eta_{m,z}\left[ \left( \mathcal{L}-L_z \right)u\right]\right\vert_{\beta,\infty}+\left\vert \langle u,\eta_{m,z}\rangle_z\right\vert_{\beta,\infty}\nonumber\\
&\leq& l\left(m\right)^{1+\beta}\left( \varepsilon \left\vert u\right\vert_{1+\beta,\infty}+C_{\varepsilon}\left\vert u\right\vert_{0}\right)+C\left(l\left(m\right)^{\beta}\left\vert f\right\vert_{0}+\left\vert f\right\vert_{\beta,\infty}\right)\nonumber\\
&&+CF\left(m,x,z\right)\left\vert u \right\vert_{1+\beta,\infty}.\label{g2}
\end{eqnarray}

An immediate conclusion of this estimate is 
\begin{eqnarray}
\left\vert \eta_{m,z} u\right\vert_{\beta,\infty}&\leq& C\left( \lambda^{-1}\wedge T\right)\left(F\left(m,x,z\right)\left\vert u\right\vert_{1+\beta,\infty}+l\left(m\right)^{\beta}\left\vert f\right\vert_{0}+\left\vert f\right\vert_{\beta,\infty}\right)\nonumber\\
&&+C\left( \lambda^{-1}\wedge T\right) l\left(m\right)^{1+\beta}\left( \varepsilon \left\vert u\right\vert_{1+\beta,\infty}+C_{\varepsilon}\left\vert u\right\vert_{0}\right),\label{g4}
\end{eqnarray}
where $C$ does not depend on $\lambda,m$. Thus,
\begin{eqnarray}
&&\left\vert  u\right\vert_{0}\leq C\sup_z \left\vert \eta_{m,z}u\right\vert_{\beta}\notag\\
&\leq& C\left( \lambda^{-1}\wedge T\right) l\left(m\right)^{1+\beta}\left(\left\vert u\right\vert_{1+\beta,\infty}+\left\vert f\right\vert_{\beta,\infty}\right),\label{g3}
\end{eqnarray}
Combining $\eqref{g1},\eqref{g2},\eqref{g3}$, we then have
\begin{eqnarray}
&&\left\vert  \eta_{m,z} u\right\vert_{1+\beta,\infty}\nonumber\\
&\leq& \varepsilon l\left(m\right)^{1+\beta}\left\vert u\right\vert_{1+\beta,\infty}+C_{\varepsilon}\left( \lambda^{-1}\wedge T\right) l\left(m\right)^{2+2\beta}\left(\left\vert u\right\vert_{1+\beta,\infty}+\left\vert f\right\vert_{\beta,\infty}\right)\nonumber\\
&&+Cl\left(m\right)^{\beta}\left\vert f\right\vert_{\beta,\infty}+CF\left(m,x,z\right)\left\vert u \right\vert_{1+\beta,\infty}.\label{g5}
\end{eqnarray}

On the other hand, by Lemma \ref{beta},
\begin{eqnarray*}
	\left\vert L^{\mu}u\right\vert_{\beta,\infty}&\leq& C\left\vert L^{\mu}u\right\vert_{\beta}\leq C l\left( m\right)^{\beta}\left\vert  L^{\mu}u\right\vert_{0}+C\sup_{z}\left\vert \eta_{m,z} L^{\mu}u\right\vert_{\beta}\\
	&\leq& C l\left( m\right)^{\beta}\left\vert  L^{\mu}u\right\vert_{0}+C\sup_{z}\left\vert \eta_{m,z} L^{\mu}u\right\vert_{\beta,\infty}.
\end{eqnarray*}
Let $\rho\left( z,y\right)=1,\nu\left( dy\right)=\mu\left( dy\right)$ in Lemma \ref{cont2} and utilize $\eqref{same}$.
\begin{eqnarray}
&&\sup_{z}\left\vert \eta_{m,z} L^{\mu}u\right\vert_{\beta,\infty}\nonumber\\
&\leq& \sup_{z}\left\vert \eta_{m,z} u\right\vert_{1+\beta,\infty}+\sup_{z}\left\vert u\left(L^{\mu}\eta_{m,z} \right)\right\vert_{\beta,\infty}+\sup_{z}\left\vert \langle u,\eta_{m,z}\rangle_z\right\vert_{\beta,\infty}\nonumber\\
&\leq& \sup_{z}\left\vert \eta_{m,z} u\right\vert_{1+\beta,\infty}+C l\left(m\right)^{1+\beta}\left( \varepsilon \left\vert u\right\vert_{1+\beta,\infty}+C_{\varepsilon}\left\vert u\right\vert_{0}\right),\label{g8}
\end{eqnarray}
where $C$ does not depend on $\lambda,m$. Combining $\eqref{g3},\eqref{g5}$, we obtain
\begin{eqnarray*}
	&&\left\vert u\right\vert_{1+\beta,\infty}\leq C\left(\left\vert u\right\vert_{0}+\left\vert L^{\mu}u\right\vert_{\beta,\infty}\right)\\
	&\leq& C\left( \lambda^{-1}\wedge T\right) l\left(m\right)^{1+\beta}\left(\left\vert u\right\vert_{1+\beta,\infty}+\left\vert f\right\vert_{\beta,\infty}\right)+C\sup_{z}\left\vert \eta_{m,z} u\right\vert_{1+\beta,\infty}\\
	&&+ C l\left(m\right)^{1+\beta}\left( \varepsilon \left\vert u\right\vert_{1+\beta,\infty}+C_{\varepsilon}\left\vert u\right\vert_{0}\right)\\
	&\leq& \varepsilon l\left(m\right)^{1+\beta}\left\vert u\right\vert_{1+\beta,\infty}+C_{\varepsilon}\left( \lambda^{-1}\wedge T\right) l\left(m\right)^{2+2\beta}\left(\left\vert u\right\vert_{1+\beta,\infty}+\left\vert f\right\vert_{\beta,\infty}\right)\nonumber\\
	&&+C_{\varepsilon}l\left(m\right)^{\beta}\left\vert f\right\vert_{\beta,\infty}+CF\left(m,x,z\right)\left\vert u \right\vert_{1+\beta,\infty}.
\end{eqnarray*}

In the inequality above, we first set $m$ sufficiently large so that 
\begin{equation*}
CF\left(m,x,z\right)\left\vert u \right\vert_{1+\beta,\infty}\leq \frac{1}{4}\left\vert u \right\vert_{1+\beta,\infty}.
\end{equation*}
For such an $m$, we then select $\varepsilon$ such that $\varepsilon l\left(m\right)^{1+\beta}<1/4$. At last, we choose $\lambda$ large enough so that for such $m,\varepsilon$, $C_{\varepsilon}\left( \lambda^{-1}\wedge T\right) l\left(m\right)^{2+2\beta}<1/4$. As a summary, with appropriate choice of $m,\varepsilon,\lambda$, $\left\vert u\right\vert_{1+\beta,\infty}\leq C\left(\lambda\right)\left\vert f\right\vert_{\beta,\infty}$. 

We need $\lambda$ to be sufficiently large though, say $\lambda\geq \lambda_0$. To completely relax this constraint, let us consider $v\left( t,x\right):=e^{\left( \lambda-\lambda_0\right)t}u\left( t,x\right),\lambda>0$, where $u$ solves $\eqref{eq3}$. Then $v$ is a solution to 
\begin{eqnarray*}
	\partial_t v\left( t,x\right)&=&\mathcal{L}v\left( t,x\right)-\lambda_0 v\left( t,x\right)+e^{\left( \lambda-\lambda_0\right)t}f\left( t,x\right), \lambda\geq 0,\\
	v\left( 0,x\right)&=& 0,\quad\left( t,x\right)\in H_T,\nonumber
\end{eqnarray*}
and 
\begin{eqnarray*}
	\left\vert v\right\vert_{1+\beta,\infty}= \left\vert e^{\left( \lambda-\lambda_0\right)t}u\right\vert_{1+\beta,\infty}\leq C_{\lambda_0}\left\vert e^{\left( \lambda-\lambda_0\right)t}f\right\vert_{\beta,\infty}.	
\end{eqnarray*}
Namely, $\left\vert u\right\vert_{1+\beta,\infty}\leq C_{\lambda_0}\left\vert f\right\vert_{\beta,\infty}$. Note $C_{\lambda_0}$ is uniform with respect to $\lambda$. 

Now we can conclude from $\eqref{g4}$, $\eqref{g3}$ and Lemma \ref{beta} that
\begin{eqnarray*}
	\left\vert u\right\vert_{\beta,\infty}&\leq& C l\left( m\right)^{\beta}\left\vert u\right\vert_{0}+C\sup_{z}\left\vert \eta_{m,z}u\right\vert_{\beta,\infty}\leq C\left( \lambda^{-1}\wedge T\right) \left\vert f\right\vert_{\beta,\infty},
\end{eqnarray*}
where $C$ does not depend on $\lambda,u,f$.

Again, according to Theorems \ref{thm2}, \ref{thm1} and $\eqref{g2},\eqref{g7}$, there is a constant $C$ depending on $\kappa,\beta,d,T,\mu,\nu$ such that for all $0\leq s<t\leq T$, $\kappa\in\left[ 0,1\right]$,
\begin{eqnarray*}
	&&\left\vert  \eta_{m,z} u\left( t,\cdot\right)-\eta_{m,z} u\left( s,\cdot\right)\right\vert_{1+\beta,\infty}\\
	&\leq& C\left( t-s\right)^{1-\kappa}\big( \left\vert u\left(L_z\eta_{m,z} \right)\right\vert_{\beta,\infty}+\left\vert\eta_{m,z}f\right\vert_{\beta,\infty}\\
	&&+\left\vert\eta_{m,z}\left[ \left( \mathcal{L}-L_z \right)u\right]\right\vert_{\beta,\infty}+\left\vert \langle u,\eta_{m,z}\rangle_z\right\vert_{\beta,\infty}\big)\\
	&\leq& Cl\left( m\right)^{1+\beta}\left( t-s\right)^{1-\kappa}\left\vert f\right\vert_{\beta,\infty}.\nonumber
\end{eqnarray*}
Apply Lemma \ref{beta} and repeat derivation $\eqref{g8}$ for the difference function,
\begin{eqnarray*}
	&&\left\vert u\left( t,\cdot\right)-u\left( s,\cdot\right)\right\vert_{1+\beta,\infty}\\
	&\leq& C\left(\left\vert u\left( t,\cdot\right)-u\left( s,\cdot\right)\right\vert_{0}+\left\vert L^{\mu}u\left( t,\cdot\right)-L^{\mu}u\left( s,\cdot\right)\right\vert_{\beta,\infty}\right)\\
	&\leq& C\sup_z\left\vert  \eta_{m,z} u\left( t,\cdot\right)-\eta_{m,z} u\left( s,\cdot\right)\right\vert_{0}\\
	&&+C l\left(m\right)^{\beta}\sup_z \left\vert  \eta_{m,z}L^{\mu} u\left( t,\cdot\right)-\eta_{m,z}L^{\mu} u\left( s,\cdot\right)\right\vert_{\beta,\infty}\\
	&\leq& \left(C l\left(m\right)^{\beta}+C_{\varepsilon}l\left(m\right)^{1+2\beta}\right)\left\vert  \eta_{m,z} u\left( t,\cdot\right)-\eta_{m,z} u\left( s,\cdot\right)\right\vert_{1+\beta,\infty}\\
	&& +\varepsilon l\left(m\right)^{1+2\beta}\left\vert u\left( t,\cdot\right)-u\left( s,\cdot\right)\right\vert_{1+\beta,\infty}.
\end{eqnarray*}	
Choose $\varepsilon$ such that $\varepsilon l\left(m\right)^{1+2\beta}<1/2$. Then we arrive at 
\begin{eqnarray*}
	&&\left\vert u\left( t,\cdot\right)-u\left( s,\cdot\right)\right\vert_{1+\beta,\infty}\\
	&\leq& \left(C l\left(m\right)^{\beta}+C_{\varepsilon}l\left(m\right)^{1+2\beta}\right)\left\vert  \eta_{m,z} u\left( t,\cdot\right)-\eta_{m,z} u\left( s,\cdot\right)\right\vert_{1+\beta,\infty}\\
	&\leq& C\left( t-s\right)^{1-\kappa}\left\vert f\right\vert_{\beta,\infty}.
\end{eqnarray*}		

Uniqueness of the solution is a direct consequence of these estimates.

\qquad\\
\noindent\textsc{Existence. }Let $\mathcal{V}\left( H_T\right)$ be the linear space that for any $v\in \mathcal{V}\left( H_T\right)$, there exists a unique $f\in\tilde{C}^{\beta}_{\infty,\infty}\left( H_T\right)$ such that $v\left( t,x\right)=\int_0^t f\left( s,x\right)ds$. Equip $\mathcal{V}\left( H_T\right)$ with norm $\left\vert v\right\vert_{\mathcal{V}}:=\left\vert f\right\vert_{\beta,\infty}$. Let $\mathcal{U}\left( H_T\right)$ be the linear space that for any $u\in \mathcal{U}\left( H_T\right)$, there is $g\in\tilde{C}^{1+\beta}_{\infty,\infty}\left( H_T\right)$ such that $u\left( t,x\right)=\int_0^t g\left( s,x\right)ds$. Endow $\mathcal{U}\left( H_T\right)$ with norm $\left\vert u\right\vert_{\mathcal{U}}:=\left\vert u\right\vert_{1+\beta,\infty}$. Then $\mathcal{V}\left( H_T\right)$ is a normed linear space and $\mathcal{U}\left( H_T\right)$ is a Banach space. Define for $\theta\in\left[0,1\right]$,
\begin{eqnarray*}
	\mathcal{T}_{\theta}u\left( t,x\right)&=&\theta\left( u\left( t,x\right)-\int_0^t \left( \mathcal{L}u\left( s,x\right)-\lambda u\left( s,x\right)\right)ds\right)\\
	&+&\left( 1-\theta\right)\left( u\left( t,x\right)-\int_0^t \left( L^{\nu} u\left( s,x\right)-\lambda u\left( s,x\right)\right)ds\right)\\
	&:=& u\left( t,x\right)-\int_0^t \left[\mathcal{L}_{\theta}u\left( s,x\right)-\lambda u\left( s,x\right)\right]ds,
\end{eqnarray*}
where $\mathcal{L}_{\theta}=\theta\mathcal{L}+\left(1-\theta \right)L^{\nu}$. Take $u\in\mathcal{U}\left( H_T\right)$. Then $u\left( t,x\right):=\int_0^t g\left( s,x\right)ds$ for some $g\in\tilde{C}^{1+\beta}_{\infty,\infty}\left( H_T\right)$. Clearly, for any $\theta\in\left[0,1\right]$, $u$ solves
\begin{eqnarray*}
	u\left( t,x\right)&=&\int_0^t \lbrack \mathcal{L}_{\theta} u\left( s,x\right)-\lambda u\left( s,x\right)+\Big( g\left( s,x\right)-\mathcal{L}_{\theta}u\left( s,x\right)+\lambda u\left( s,x\right)\Big)\rbrack ds.
\end{eqnarray*}
Therefore,
\begin{eqnarray*}
	\mathcal{T}_{\theta}u\left( t,x\right)&=&\int_0^t \lbrack  g\left( s,x\right)-\mathcal{L}_{\theta}u\left( s,x\right)+\lambda u\left( s,x\right)\rbrack ds,
\end{eqnarray*}
where by Lemma \ref{cc}, Proposition \ref{ppr1} and Corollary \ref{coo},
\begin{eqnarray*}
	\left\vert \mathcal{T}_{\theta}u \right\vert_{\mathcal{V}}=\left\vert g-\mathcal{L}_{\theta}u+\lambda u \right\vert_{\beta,\infty}\leq C \left\vert  u \right\vert_{1+\beta,\infty}<\infty.
\end{eqnarray*}
Then, $\mathcal{T}_{\theta}\left[\mathcal{U}\left( H_T\right)\right]\subset \mathcal{V}\left( H_T\right)$. Meanwhile, by estimates we derived above, there is $C$ independent of $u, \theta$ such that 
\begin{eqnarray*}
	\left\vert u \right\vert_{\mathcal{U}}=\left\vert u \right\vert_{1+\beta,\infty}\leq C\left\vert g-\mathcal{L}_{\theta}u+\lambda u \right\vert_{\beta,\infty}\leq C \left\vert \mathcal{T}_{\theta}u \right\vert_{\mathcal{V}}.
\end{eqnarray*}

Theorem \ref{thm2} says $T_0$ maps $\mathcal{U}$ onto $\mathcal{V}$. By Theorem 5.2 in \cite{gt}, so does $T_1$.

\section{Appendix}

\begin{lemma}\cite[Lemma 2]{mf}\label{ess}
	Let $\nu$ be a L\'{e}vy measure and $w$ be the scaling function which $\nu$ satisfies \textbf{A(w,l)} for. Then,\\
	\noindent a) there are constants $C_1, C_2>0$ such that 
	\begin{eqnarray}
	C_1\varsigma\left( r\right)\leq w\left( r\right)^{-1}\leq C_2 \varsigma\left( r\right), \quad\forall r>0.
	\end{eqnarray}
	\noindent b) $\int_{\left\vert y\right\vert \leq 1}w\left( \left\vert y\right\vert
	\right) \nu\left( dy\right)=+\infty$.\\
	\noindent c) For any $\varepsilon>0$, $\int_{\left\vert y\right\vert \leq 1}w \left( \left\vert y\right\vert
	\right) ^{1+\varepsilon }\nu\left( dy\right) <\infty$.\\
	\noindent d) For any $\varepsilon >0$, $\int_{\left\vert y\right\vert \leq 1}\left\vert y\right\vert ^{\varepsilon
	}w\left( \left\vert y\right\vert \right) \nu\left( dy\right) <\infty$.
\end{lemma}

\begin{lemma}\cite[Lemma 5]{cr}\label{lemma3}
	Let $\nu$ be a L\'{e}vy measure satisfying \textbf{A(w,l)}. $Z^{\tilde{\nu}_R}_t$ is the L\'{e}vy process associated to $\tilde{\nu}_R,R>0$. For each $t, R$, $Z^{\tilde{\nu}_R}_t$ has a bounded and continuous density function $p^R\left( t,x\right),t\in\left( 0,\infty\right), x\in\mathbf{R}^d$. And $p^R\left( t,x\right)$ has bounded and continuous derivatives up to order $4$. Meanwhile, for any multi-index $\left\vert\vartheta\right\vert\leq 4$, 
	\begin{eqnarray*}
		\int\left\vert \partial^{\vartheta}p^R\left( t,x\right)\right\vert dx &\leq&  C\gamma\left( t\right)^{-\left\vert\vartheta\right\vert},\\
		\sup_{x\in\mathbf{R}^d}\left\vert \partial^{\vartheta}p^R\left( t,x\right)\right\vert  &\leq&  C\gamma\left( t\right)^{-d-\left\vert\vartheta\right\vert},
	\end{eqnarray*}
	where $C>0$ is independent of $t,R$. For any $\beta\in\left(0,1\right)$ such that $\left\vert\vartheta\right\vert+\beta<4$,
	\begin{eqnarray*}
		\int\left\vert \partial^{\beta}\partial^{\vartheta}p^R\left( t,x\right)\right\vert dx &\leq&  C\gamma\left( t\right)^{-\left\vert\vartheta\right\vert-\beta}.
	\end{eqnarray*}
	For any $a>0$, there is a constant $C>0$ independent of $t,R$, so that 
	\begin{eqnarray*}
		\int_{\left\vert x\right\vert>a}\left\vert \partial^{\vartheta}p^R\left( t,x\right)\right\vert dx &\leq&  C\left( \gamma\left( t\right)^{2-\left\vert\vartheta\right\vert}+t\gamma\left( t\right)^{-\left\vert\vartheta\right\vert}\right). 
	\end{eqnarray*}
\end{lemma}	

\section*{Acknowledgments} 
I would like to thank Prof. Remigijus Mikulevi\v{c}ius for useful discussions.

\end{document}